\documentclass[11pt]{amsart}
\usepackage{pdfsync} 
\usepackage{amsmath} 
\usepackage{lmodern}
\usepackage[T1]{fontenc}
\usepackage[english]{babel}
\usepackage[utf8]{inputenc}
\usepackage{geometry}
\usepackage[all]{xy}  

\usepackage{amssymb}
\usepackage[hypcap]{caption}
\usepackage{amsmath, amsthm, amsfonts, amscd}
\usepackage[linesnumbered,lined,boxed,commentsnumbered]{algorithm2e}
\usepackage{xcolor}
\usepackage{ latexsym }
\usepackage{mathrsfs}
\usepackage{float}
\usepackage{graphicx}
\usepackage[labelformat=empty]{caption}
\usepackage{microtype}
\usepackage{todonotes}
\usepackage{verbatim}
\usepackage{hyperref}

\usepackage{mdwlist}
\usepackage{mathabx}
\usepackage{mathtools}

\usepackage[utf8]{inputenc}
\usepackage[english]{babel}
\usepackage{tikz-cd} 
\usepackage{tikz}
\usetikzlibrary{decorations.markings, decorations.pathreplacing, positioning,arrows, matrix, decorations.pathmorphing, shapes.geometric, calc}
\usetikzlibrary{backgrounds, decorations}

\newtheorem{theorem}{\scshape{Theorem}}[section]

\newcommand{\mc}[0]{\mathcal}
\newtheorem{cor}[theorem]{\scshape{Corollary}}

\newtheorem{proposition}[theorem]{\scshape{Proposition}}

\newtheorem{lemma}[theorem]{\scshape{Lemma}}

\newtheorem*{theorem*}{Theorem}

\newtheorem*{prop*}{Proposition}
\newtheorem*{cor*}{Corollary}

\theoremstyle{definition}

\newtheorem{definition}[theorem]{\scshape{Definition}}

\newtheorem{example}[theorem]{\scshape{Example}}

\newtheorem{remark}[theorem]{\scshape{Remark}}

\newtheorem{thmletter}{Theorem}

\newcommand{\Length}{\textsc{L}}
\newcommand{\ab}{\textsf{ab}}
\newcommand{\abs}{\textsf{abs}}
\newcommand{\mbb}{\mathbb}

\title{Group equations with abelian predicates}
\author{Laura Ciobanu and Albert Garreta}
\address{School of Mathematical and Computer Sciences, Heriot-Watt University,
Edinburgh EH14 4AS, UK}

\address{Basque Center of Applied Mathematics, Bilbao, Spain}

\email{l.ciobanu@hw.ac.uk}

\email{garreta.a@gmail.com}



\begin{document}

\maketitle



\begin{abstract}
In this paper we begin the systematic study of group equations with abelian predicates in the main classes of groups where solving equations is possible. We extend the line of work on word equations with length constraints, and more generally, on extensions of the existential theory of semigroups, to the world of groups. 

We use interpretability by equations to establish model-theoretic and algebraic conditions which are sufficient to get undecidability. We apply our results to (non-abelian) right-angled Artin groups, and show that the problem of solving equations with abelian predicates is undecidable for these. We obtain the same result for  hyperbolic groups whose abelianisation has torsion-free rank at least two.
By contrast, we prove that in groups with finite abelianisation, the problem can be reduced to solving equations with recognisable constraints, and so this is decidable in right-angled Coxeter groups, or more generally, graph products of finite groups, as well as hyperbolic groups with finite abelianisation. 
\end{abstract}

\section{Introduction}

 The study of equations in algebraic structures such as semigroups, groups or rings is a fundamental topic in mathematics and theoretical computer science that finds itself at the frontier between decidable and undecidable. In this paper we draw inspiration from the extensive work on word equations with various length constraints \cite{Abdulla, RichardBuchi1988, DayManeaWE, GarretaGray, majumdar2021quadratic} 
 and consider equations in groups with similar (non-rational) constraints. The term `word equations' refers to equations in free monoids, and
 deciding algorithmically whether a word equation has solutions satisfying certain linear length constraints is a major open question: it has deep implications, both theoretical (if undecidable, it would offer a new solution to Hilbert's 10th problem about the satisfiability of polynomial equations with integer coefficients) and practical, in the context of string solvers for security analysis. 
 We refer the reader to the surveys \cite{amadini,ganesh} for an overview of the area from several viewpoints, of both theoretical and applied nature.
 

 We find similarities to word equations with length constraints, but also new territory, when entering the world of groups. The question of decidability of (systems of) equations is widely known as the {\em Diophantine Problem}, and we denote it by $\mc{DP}$, or by $\mc{DP}(G)$ when it refers to a (semi)group $G$ (see Section \ref{sec:DP}).  We extend the $\mc{DP}$ by requiring that the solution sets satisfy certain (non-rational) constraints, and study this augmented problem in some of the most important classes of infinite groups for which the $\mc{DP}$ is known to be decidable, such as right-angled Artin groups (also called partially commutative groups) and more generally, graph products of groups, as well as hyperbolic groups. Our motivation is twofold: first, explore extensions of the $\mc{DP}$ that have not been systematically studied for groups before, and second, develop algebraic and model-theoretic tools to supplement the combinatorial techniques used for solving word equations with length constraints. By comparison, group equations with \emph{rational} constraints have been studied for over 20 years, and they proved to have remarkable applications to other decision problems in groups, or to understanding the solution sets of equations as formal languages \cite{eqns_hyp_grps, DahmaniIJM09, DG, dgh01}. 
  
 Among several types of length constraints (see Section \ref{DP:constraints}), we focus on abelian predicates (such as $\textsf{AbelianEq}$ in \cite{DayManeaWE}), which we also call \emph{abelianisation constraints}, and denote by $\ab$. These constraints require that the `abelian form' of the solutions, where any two symbols commute, satisfies a given set of equations as well; equivalently, the constraints can be viewed as equations in the abelianisation of $G$. We denote this problem by $(\mc{DP}, \ab)$. When referring to a specific structure $G$ we use $\mc{DP}(G,\ab)$, and note that this can be seen as introducing an abelian predicate or relation to the existential theory of $G$. 
 Our main results are:
 
 \begin{thmletter}(Theorem \ref{thm:DP_RAAGS})
	Let $G$ be a right-angled Artin group that is not abelian. Then $\mc{DP}(G, \ab)$ is undecidable.
\end{thmletter}
 \begin{thmletter}[\ref{thm:hyp}]
Let $G$ be a hyperbolic group with abelianisation of torsion-free rank $\geq 2$. Then $DP(G, ab)$ is undecidable. 
 \end{thmletter}
 
 The next theorem is a counterweight to the previous two, as it gives positive results when the abelianisation is finite. 

\begin{thmletter}(Theorems \ref{thm:finiteab}, \ref{thm:hyp_finiteab}, \ref{thm:graph_prod})
Let $G$ be a group where the $\mc{DP}$ with recognisable constraints is decidable. Then $\mc{DP}(G, \ab)$ is decidable. In particular, this holds if:
\begin{enumerate}
\item[1.] $G$ is a hyperbolic group with finite abelianisation.
\item[2.] $G$ is a graph product of finite groups, such as a right-angled Coxeter group.
\end{enumerate}
\end{thmletter}
 
The decidability of $\mc{DP}(G, \ab)$ when $G$ has decidable $\mc{DP}(G)$ and its abelianisation has torsion-free rank $1$ remains an open problem. This applies, for example, to hyperbolic groups with first Betti number equal to $1$, or to certain classes of one-relator groups, including dihedral Artin groups and soluble Baumslag-Solitar groups (see Section \ref{sec:conclusions}).

The paper interleaves algebra and model theory, and introduces tools that we hope will lead to results for further classes of groups and possibly semigroups, and more diverse length constraints. 
 The main tool used to prove undecidability results in this paper is \emph{interpretability} using disjunctions of equations (i.e.\ positive existential formulas), or \emph{PE-interpretability}, which allows us to translate one structure into another and to reduce the Diophantine Problem from one structure to the other (Section \ref{sec:interpretability}). The main reductions are to the Diophantine Problem in the ring of integers, which is a classical undecidable problem (Hilbert's 10th problem). We give a technical result in Section \ref{sec:technical} which enables us to encode the Diophantine Problem over the integers into the Diophantine Problem with abelianisation constraints in a group, assuming the group satisfies certain model-theoretic properties. This technical result (Lemma \ref{l: technical_lemma}) is then applied to right-angled Artin groups and to hyperbolic groups with large abelianisation.

 Our starting point is B\"uchi and Senger's well-known paper \cite{RichardBuchi1988}, where they show that (positive) integer addition and multiplication can be encoded into word equations in the free semigroup $\Sigma^*$ on a finite alphabet $\Sigma$, when requiring that the solutions satisfy an abelian predicate; by the undecidability of Hilbert's 10th problem, such enhanced word equations are also undecidable.
  In order to encode the multiplication in the ring $(\mbb{Z}, \oplus, \odot)$ within a group or semigroup, they need two `independent' elements, which can be taken to be two of the free generators if the (semi)group is free. In non-free groups we require the existence of two elements that play a similar role; however, it is not enough to pick two generators, these elements also need to satisfy additional properties with respect to the abelianisation. 
   Finding such a pair of elements, which we call \emph{abelian-primitive}, is difficult or impossible for arbitrary groups, but in this paper we show that they can be found in right-angled Artin groups and many hyperbolic groups. A large part of the paper is concerned with defining, studying and finding abelian-primitive generators, which are not well-defined in general.

The outline of the article is as follows.
 In Section \ref{sec:Prelim} we introduce notions of length and exponent-sum functions associated to a generator in a group. We describe equations in groups and the Diophantine Problem, together with several variants; in the variants the solutions have to satisfy not just a system of equations in a chosen group, but also certain constrains, such as rational, recognisable, or a set of equations in the abelianisation (Section \ref{DP:constraints}). We introduce \emph{definability} by equations, or \emph{PE-definability}, and \emph{interpretability} using equations, or \emph{PE-interpretability}, in Section \ref{sec:interpretability}.
 This allows us to translate one structure into another and to reduce the $\mc{DP}$ from one structure to the other (Proposition \ref{RedCor}).  
 In Section \ref{sec:free} we give results for free groups and semigroups following B\"uchi and Senger; although these results might be known to the experts and get generalised in later sections, it is useful to see the basic idea in action for free groups. 
 
 Section \ref{sec:fin_ab} provides an interesting contrast to the rest of the paper, in that we obtain positive results for $\mc{DP}(G,\ab)$ for cases when $G$ has finite abelianisation. In this case the $\mc{DP}(G,\ab)$ can be reduced to the $\mc{DP}(G)$ with recognisable constraints, and we showcase groups for which these problems are decidable. 
 Section \ref{sec:technical} provides the technical tools (Lemma \ref{l: technical_lemma}) required to encode the Diophantine Problem over the integers into a group Diophantine Problem with abelianisation constraints, if the group satisfies certain model-theoretic properties. We use this technical section to establish the most involved result of the paper (Theorem \ref{thm:DP_RAAGS}) in Section \ref{sec:RAAGs}, namely, that the $\mc{DP}$ with abelianisation constraints is undecidable in (non-abelian) partially commutative groups. In Section \ref{sec:hyp}  we establish the main result for hyperbolic groups, and in the final section we outline several of the many possible directions of future research that naturally stem from this paper.

\section{Preliminaries}
\label{sec:Prelim}
Let $\Sigma$ be a finite set, and let $\Sigma^*$ be the free monoid with basis $\Sigma$ consisting of the set of all (finite) words on $\Sigma$, and including the empty word. Let $||w||$ denote the word length of any $w \in \Sigma^*$. For any $a\in \Sigma$, let $|| w ||_a$ count the number of occurrences of $a$ in $w$; for example, $||abab^2||_a=2$.
Define $\Sigma^{-1}$ as the set of formal inverses of elements in $\Sigma$ and denote the free group with generating set $\Sigma$ by $F(\Sigma)$.


 
 For any group $G$, let $\ab:G\to G^{\ab}$ be the natural abelianisation map to $G^{\ab}= G/G'$, that is, the quotient of $G$ by its commutator subgroup; $G^{\ab}$ is a commutative group which will decompose into an infinite part of the form $\mathbb{Z}^m$, for some $m\geq 1$, and a finite abelian group $H$. The integer $m$ is \emph{the free rank} or \emph{first Betti number} $b_1(G)$ of $G^{\ab}$, and we write $b_1(G)=m$. 
 
 The analogous map $\ab: \Sigma^* \mapsto \mbb{N}^{|\Sigma|}$ for free monoids is called the \emph{Parikh map} in language theory. 

\noindent\textbf{Right-angled Artin groups (RAAGs).} A class of groups that lie between the free (non-abelian) and the free abelian groups are the \emph{right-angled Artin groups} (RAAGs), and we will use this short-hand notation henceforth. The most common way of describing a RAAG is via a finite undirected graph $\Gamma$ with no auto-adjacent vertices (i.e. no loops at any vertex) and no multiple edges between two vertices, and letting the vertices of $\Gamma$ be the generators of the RAAG $G\Gamma$ based on $\Gamma$. The relations between generators correspond to the edges: for every edge $(u,v)$ in $\Gamma$ we introduce the commuting relation $uv=vu$.
We will often write $G$ instead of $G\Gamma$ when the graph $\Gamma$ is unambiguous.
\begin{example}\label{ex:3edgelinesegment}
Let $\Gamma$ be the graph below with vertices $\{a,b,c,d\}$ and $3$ edges. The RAAG $G=G\Gamma$ based on $\Gamma$ has the presentation $\langle a,b,c,d \mid [a,b]=[b,c]=[c,d]=1 \rangle$. 
 \[
  \xymatrix{\stackrel{a}{\bullet}\ar@{-}[r] &\stackrel{b}{\bullet} \ar@{-}[r]&\stackrel{c}{\bullet}\ar@{-}[r] &\stackrel{d}{\bullet} }.
 \]

 \end{example}

\noindent\textbf{Length and exponent-sum.} Every element $g$ in a group $G$ with finite generating set $\Sigma$ has a length $|g|_{\Sigma}$, which is the word length $||.||$ of a shortest word $w$ representing $g$ in $G$. For example, the length of $aba^{-1}$ in $\mbb{Z}^2$ with respect to $\{a, b\}$ is $|aba^{-1}|_{\Sigma}=||b||=1$.

For any generator $x$ of $F(\Sigma)$, the map $| . |_x: F(\Sigma) \to \mbb{Z}$ represents the \emph{exponent-sum} of $x$ in a word $w$; that is, $|w|_x =||w||_x -||w||_{x^{-1}}$, so for example $|xyx^{-1}y^2|_x=0$. One can define the exponent-sum of a generator $x$ in an element $g$ for certain (but not all) groups beyond free groups, and then use the same notation $|g|_x$, as we explain next.
The length and the image under abelianisation are well-defined for any element in any finitely generated group. However, the situation regarding the exponent-sum (of a generator) is more complicated. 
For example, if $H$ is a group and $x\in H$ a generator of order $5$, then one may claim the exponent-sum of $x$ in the element $x^3$ to be $3$; but $x^3=x^{-2}$ in $H$, and in $x^{-2}$ the exponent-sum of $x$ appears to be $-2$.

We give below sufficient conditions for the exponent-sum function to be well-defined. Among other requirements, we need the groups to be infinite and have infinite abelianisation, and moreover, the generators for which we consider the exponent-sums to have infinite order not just in the group, but also in the abelianisation.

\def\ap{abelian-primitive}
\begin{definition}\label{def:ap}
Let $G$ be a group generated by a set $S$, and let $\ab:G\to G^{\ab}$ be the natural abelianisation map. Suppose $b_1(G^{\ab})\geq 1$, that is, $G^{\ab}$ is infinite. 

A subgroup $K \leq G^{\ab}$ is a $\mbb{Z}$ - factor if $K$ is isomorphic to $\mbb{Z}$ and $G^{\ab} = K \times H$ for some $H\leq G^{\ab}$.
If $s\in S$ is such that $\ab(s)$ generates a $\mbb{Z}$ - factor in $G^{\ab}$, we say that $s$ is \emph{\ap}.

Note that any \ap \ element has infinite order in $G$.
\end{definition}
\begin{example}\label{ex:ap}
All generators of a free group are \ap. More generally, any group element corresponding to a vertex in the defining graph of a RAAG is \ap.
\end{example}
We extend the definition of exponent-sum from free groups to any group $G$ containing \ap \ elements. Let $\pi$ be the natural projection $\pi: F(S) \to G$, $s \in S$ be \ap, and $\abs(x)$ be the absolute value of $x \in \mathbb{Z}$ (we use $\abs$ to avoid confusion with the notation for length). For $g\in G$ we let the \emph{exponent-sum} of $s$ in $g$ be
\begin{align*}|g|_s :=  |w_g|_s, &\textrm{ where } w_g\in F_S \textrm{ is such that }\\ &\abs(|w_g|_s) \textrm{ is minimal among } w \in F(S) \textrm{ such that } \pi(w)= g.\end{align*}
%

The next lemma shows that the exponent-sum function is well-defined.
\begin{lemma}\label{l: well_def}
   Let $G$ be generated by $S$, and suppose $w_g, w_g'\in F(S)$ are such that $w_g\neq w_g'$, $\pi(w_g)=\pi(w_g')=g$, $\abs(|w_g|_s) = \abs(|w_{g'}|_s)$ is minimal, and $s$ is \ap. 
   
   Then $|w_g|_s = |w_{g'}|_s$.
\end{lemma}
\begin{proof}
 Suppose $|w_g|_s \neq |w_{g'}|_s$. Then $|w_g|_s = - |w_{g'}|_s  \neq 0$. From the equality $\ab(\pi(w_g)) = \ab(\pi(w_{g'}))$, solving for $\ab(s)$ in $G^{\ab}$, one obtains $( \ab(s))^t \in H$ for some subgroup $H$ such that $G^{\ab}= \langle \ab(s) \rangle \times H$,
contradicting the fact that $s$ is \ap.
\end{proof}

\begin{remark}\label{l:product}
In later sections of the paper we will need the following observation for $G$ with torsion-free abelianisation rank $b_1(G)\geq 2$. Suppose $s_1, s_2 \in G$ are \ap \ generators such that $\langle \ab(s_1)\rangle \cap \langle \ab(s_2)\rangle =1$, or equivalently, $\langle \ab(s_1), \ab(s_2) \rangle \cong \mathbb{Z}^2$ in $G/G'$. Then $s_1s_2$ has infinite order in $G$ and $\ab(s_1s_2)$ has infinite order in $G/G'$. 

This follows immediately from noting that were $s_1s_2$ to have finite order, then so would $\ab(s_1s_2)$. But $\langle \ab(s_1), \ab(s_2) \rangle$ cannot contain a torsion element, and since $\ab(s_1s_2) \in \langle \ab(s_1), \ab(s_2) \rangle$, this contradicts $\ab(s_1s_2)$ being torsion.
\end{remark}

%


The above definition of $|g|_s$, altough a generalisation of the exponent-sum in a free monoid or group, may seem rather unnatural due to the presence of the absolute value function. Next we provide an equivalent, but more natural definition. All arguments in this paper use this equivalent formulation for the notion of being abelian-primitive.

\begin{lemma}\label{l: equivalences}

Let $G=\langle S\rangle$ be a group generated by a set $S$, and let $s\in S$ be \ap, so that $G^{\ab} \cong \langle \ab(s)\rangle \times H \cong \mbb{Z} \times H$. 
Then for each $g\in G$ there exists a unique  $t_g\in \mbb{Z}$ and a unique $c_g\in G$ such that  $g=s^{t_g} c_g$, 
  $\ab(c_g) \in H$, and $t_g = |g|_s$.

That is, for all $g\in G$, the exponent-sum $|g|_s$ of $s$ in $g$ is precisely the exponent of $\ab(s)$ in the natural projection of $g$ onto the abelianisation $G^{\ab}=\langle \ab(s)\rangle \times H$.

\end{lemma}
\begin{proof}
%

Let $\pi:F(S)\to G$ be the natural projection of the free group $F(S)$  onto $G$. Throughout the proof, an element $s\in S$ ``seen in'' $G$ will be denoted $\pi(s)$. 

Let $\ab': F(S) \to \mbb{Z}^{|S|}$ be the natural projection of $F(S)$ onto its abelianisation, and similarly let $\ab: G\to G/G'$. Additionally, let $\pi': \mbb{Z}^{|S|} \to G/G'$ be an  homomorphism such that $\pi' \circ ab' = ab \circ \pi$ (where composition is applied from right to left). We claim that $\pi'$ sends the subgroup $\langle \ab'(s) \rangle$ isomorphically onto $\langle \ab(\pi(s)) \rangle$. 
To prove the claim, note that the homomorphism $\pi'|_{\langle \ab'(s) \rangle}$ is onto. It suffices to prove that it is one-to-one. But this is clear because if $\ab'(s)^i$ belongs to the kernel of this homomorphism, then $\pi(s)^i$ belongs to the kernel of $ \ab$, which forces $i=0$ since $\ab(\pi(s))$ has infinite order in $G/G'$ because $s$ is \ap. The claim is proved.

Now let $g\in G$. Let $w_g\in F(S)$ be such that $|w_g|_s = |g|_s$.   It is clear from the definition of \ap \ elements that there exist unique $t_g\in\mbb{Z}$ and $c_g\in G'$ such that  $g=\pi(s)^{t_g} c_g$, with $\ab(c_g) \in H$, where $H\leq G/G'$  is such that $G/G' = \langle \ab(\pi(s))\rangle \times H$. 

We want to see that $t_g = |w_g|_s$. Indeed, it follows from the claim above and the definition of $t_g$ that the exponent sum of $\ab'(s)$ in the element $\ab'(w_g)\in \ab'(F(S))$ is precisely $t_g$, which can only occur if $|w_g|_s = t_g$.
%
%
%
%
%
\end{proof}

The following lemma follows immediately from Lemma \ref{l: equivalences} and will be used extensively, often without precise reference. It allows us to compute $|g|_s$ by expressing $g$ as a product of generators and their inverses, and then summing over the exponent-sums of $s$ in such a product. The resulting value is independent of the chosen product for $g$.  

\begin{lemma}[$|\cdot|_s$ is a homomorphism]\label{l: morphism}
    Let  $G=\langle S\rangle$ be a group, and let $s\in S$ be \ap.  Then $|\cdot|_s$ defines a group homomorphism $|\cdot|_s: G\to \mbb{Z}$, that is
    \begin{equation} \label{ExpEquality}
    |gh|_s = |g|_s + |h|_s
    \end{equation}
    for all $g,h\in G$.
\end{lemma}

\subsection{The Diophantine problem in groups and semigroups}\label{sec:DP}

Let $\mathbf{x}=\{X_1, \dots, X_m\}$ be a set of variables, where $m\geq 1$.
For a group $G$, a \textit{finite system of equations} in $G$ over the variables $\mathbf{x}$ is a
		finite subset $\mathcal{E}$ of the free product $G \ast F(\mathcal{\mathbf{x}})$, where
		\(F(\mathbf{x})\) is the free group on \(\mathbf{x}\). If
		\(\mathcal{E} = \{w_1, \ \ldots, \ w_n\}\), then a \textit{solution} to the
		system \(w_1 = \cdots = w_n = 1\) is a homomorphism \(\phi \colon G \ast F(\mathbf{x}) \to
		G\), such that \(\phi(w_1) = \cdots =\phi (w_n) = 1_G\) and
		 \(\phi(g) = g\) for all \(g \in G\). If $\mathcal{E}$ has a solution, then it is \emph{satisfiable}.
		 
	\begin{example} Consider the system $\mathcal{E}=\{w_1, w_2\}\subset F(a,b) \ast F(X_1, X_2)$ over the free group $F(a,b)$, where $w_1=X_1^2(abab)^{-1}$, $w_2=X_2X_1X_2^{-1}X_1^{-1}$; we set $w_1=w_2=1$, which can be written as $X_1^2=abab, X_2X_1=X_2X_1$. The solutions are $\phi(X_1)=ab, \phi(X_2)=(ab)^k$, $k \in \mbb{Z}$.
	\end{example}

For a group $G$, we say that systems of equations over $G$
are \emph{decidable} over $G$ if there is an algorithm to determine whether any
given system is satisfiable.  The question of decidability of (systems of) equations is called the
{\em Diophantine Problem} for $G$, and denoted $\mc{DP}(G)$.

\subsection{The Diophantine Problem with constraints}\label{DP:constraints}

Let $\Sigma$ be a finite alphabet and let $S=\Sigma^{\pm 1}$. Suppose $G$ is a group generated by $S$, and let $\pi: S^* \rightarrow G$ be the natural projection from the free monoid $S^*$ generated by $S$ to $G$, taking a word over the generators to the element it represents in the group. 

A language is \emph{regular} if it is recognised by a finite state automaton, as is standard.

\begin{definition}\label{def:rat_reg_con}
\ \begin{enumerate}
\item[(1)] 
A subset $L$ of $G$ is {\em recognisable} if the full preimage
$\pi^{-1}(L)$ is a regular subset of $S^*$.
\item[(2)] A subset $L$ of $G$ is {\em rational} if $L$ is the
image $\pi(R)$  of a regular subset $R$ of $S^*$.
\end{enumerate}
\end{definition}
It follows immediately from the definition that recognisable subsets of $G$ are
rational.

Let $\mathcal E$ be a system of equations on variables $\mathbf{x}=\{X_1, \dots, X_k\}$ in a (semi)group $G$.
The \emph{Diophantine Problem with rational or recognisable constraints} asks about the existence of solutions to $\mathcal E$, with some of the variables restricted to taking values in specified rational or recognisable sets. 

We next attach three types of non-rational constraints to the Diophantine problem:

\begin{enumerate}
\item 
We let $\mc{DP}(G, \Length$) denote the $\mc{DP}$ with \emph{linear length constraints}. A set of linear length constraints is a system $\Theta$ of linear integer equations where the unknowns correspond to the lengths of solutions to each variable $X_i \in \mathbf{x}$. Then $\mc{DP}(G, \Length)$ asks whether solutions to $\mathcal E$ exist for which the lengths satisfy the system $\Theta$.

\item 
Suppose the generators of $G$ can be enumerated as $S=\{g_i \mid 1 \leq i \leq m\}$, with each $g_i$ being \ap, 
and write $| . |_i$ for $| . |_{g_i}$.

We let $\mc{DP}(G,\{ | . |_i\}_i)$ denote the $\mc{DP}$ with \emph{linear exponent-sum constraints}.  A set of exponent-sum constraints is a system $\Theta$ of linear Diophantine equations where the unknowns correspond to the exponent-sums of generators in the solutions to each variable $X_i \in \mathbf{x}$. Then $\mc{DP}(G, \{| . |_i\}_i)$ asks whether a solution to $\mathcal E$ exists satisfying the system $\Theta$. 

This definition can be extended to the case where only a strict subset of $S$ are \ap\ by using unknowns in $\Theta$ that only correspond to \ap \ generators.

\item We write $\mc{DP}(G, \ab)$ for the $\mc{DP}$ where an abelian predicate is added, or equivalently, \emph{abelianisation constraints} are imposed. A set of abelianisation constraints is a system $\Theta$ of equations in the group $G^{\ab}$, and $\mc{DP}(G, \ab)$ asks whether a solution to $\mathcal E$ exists such that the abelianisation of the solution satisfies the system $\Theta$ in $G^{\ab}$.


\end{enumerate}

\begin{example} Consider the equation $XaY^2bY^{-1}=1$ over variables $X,Y$ 
in the free group $F(a,b)$ with length function $|.|=|.|_{\{a,b\}}$ and abelianisation $(\mbb{Z}^2,+)$.
\begin{enumerate}
\item An instance of $\mc{DP}(F(a,b), \Length)$ is: decide whether there are any solutions $(x,y)$ such that $|x|=|y|+2$; the answer is yes, since $(x,y)=(b^{-1}a^{-1}, 1)$ is a solution with $|x|=|y|+2$.
\item An instance of $\mc{DP}(F(a,b), \{|.|_a, |.|_b\})$ is: decide whether there are any solutions $(x,y)$ such that $|x|_a=2|y|_a+|y|_b$ and $|x|_b=3|y|_b$; the answer is `no' by solving a basic linear system over the integers with variables $|x|_a, |x|_b, |y|_a, |y|_b$ and we leave this as an exercise.
\item An instance of $\mc{DP}(F(a,b), \ab)$ is: decide whether there are any solutions $(x,y)$ such that $\ab(x)=3\ab(y)$ (we use additive notation for $\mbb{Z}^2$); the answer is no, since $\ab(xax^2by^{-1})=\ab(x)+\ab(y)+(1,1)=(0,0)$ together with $\ab(x)=3\ab(y)$ lead to $4\ab(y)=(-1,-1)$, which is not possible in $\mbb{Z}^2$.
\end{enumerate}
\end{example}

Next we provide a brief formal and unifying description of the notions of ``constraints'' and ``relations''.
A relation $r$ in a group $G$ can be specified as a subset of $S_r\subseteq G^{n_r}$. Then, a tuple $(g_1, \dots, g_r)$ satisfies the relation $r$ if and only if $(g_1, \dots, g_r) \in S_r$. In general, one can consider a set of relations $\mc{R}$ in a group $G$, and study the so-called \emph{Diophantine Problem in $G$ with  $\mc{R}$-constraints}, denoted $\mc{D}(G,\mc{R})$. The instances to this problem are \emph{systems of equations with constraints in $(G, \mc{R})$}, that is,  systems of equations in $G$ on variables $\mathbf{x}$ together with finitely many formal expressions of the form $r(w_1(\mathbf{x}), \dots, w_{n_r}(\mathbf{x}))$, where $r\in \mc{R}$, $n_r$ is the arity of $r$, and the $w_i(\mathbf{x})$ are words on $\mathbf{x}$ and possibly constant elements from $G$. A solution to such an instance is a solution to the system of equations such that all expressions $r(w_1(\mathbf{x}), \dots, w_{n_r}(\mathbf{x}))$ are true (i.e.\ the corresponding tuples of elements belong to $S_r$) after replacing each variable by its corresponding value. 
For example, the three problems $\mc{DP}(G,L), \mc{DP}(G, \{|\cdot|_i\}_i)$ and $\mc{DP}(G,\ab)$ can be formulated in this manner; for  $\mc{DP}(G,\ab)$ we think of $\ab$ as the relation determined by the set $S_{\ab} = \{(g,h)\in G^2\mid \ab(g)=\ab(h)\}$.

The following will be used later in the paper.
\begin{lemma}\label{l: DPab_in_direct_products}
	Let $G_1$ and $G_2$ be two groups, and suppose that $\mc{DP}(G_1, \ab)$ is undecidable. Then $\mc{DP}(G_1 \times G_2, \ab)$ is also undecidable.
\end{lemma}

\begin{proof}
	 Let $G=G_1\times G_2$.  Let $\pi_1$ and $\pi_2$ be the natural projections of $G$ onto $G_1$ and $G_2$. Let  $w(x_1, \dots, x_n)=1$ be an equation in $G$. Then $g_1, \dots, g_n\in G$ forms a solution to such an equation if and only if $w(\pi_i(g_1), \dots, \pi_i(g_n)) = 1_{G_i}$ for both $i=1,2$. Since  $G/G' \cong G_1/G_1' \times G_2/G_2'$, an analogous statement holds for any equation in the abelianisation $G/G'$. The lemma follows immediately from these observations.
\end{proof}

\subsection{Interpretability}\label{sec:interpretability}

Here we define \emph{interpretability by positive existential formulas} (PE-interpretability). 
This gives a partial order between algebraic structures which implies reducibility of the Diophantine Problem from one structure to the other. 
A \emph{positive existential formula} in a language $L$ is a first-order formula which can be written using only existential quantifiers, disjunctions, conjunctions, equality, and the symbols from $L$ (possibly with an extended set of constants).  It is well known that such a formula is logically equivalent to a disjunction of systems of equations. This explains the terminology used in Definition \ref{d: e-def}. This equivalence will be used extensively throughout the paper without further referring to it, and we will sometimes identify the concepts ``positive existential formula'' with ``disjunction of systems of equations''.

We restrict ourselves to groups and rings, although all definitions can be generalised to arbitrary algebraic structures, and make the convention that all logical formulas are allowed to use any number of constants from the group or ring considered.

Given a group $G$, we will denote its operation by $\cdot_G$, and its identity element by $1_G$. We will use non-cursive boldface letters to denote tuples of elements: e.g.\ $\mathbf{a}=(a_1, \dots, a_n)$. 

A \emph{relation} $r$ in $G$ is a subset of $G^n$ for some $n>0$. The number $n$ is called the \emph{arity} of $r$. A tuple of elements $(g_1, \dots, g_n)$ is said to \emph{satisfy} the relation if it belongs to $r$.


\begin{definition}\label{d: e-def} 
Let $G$ be a group and let $\mc{R}$ be a set of relations in $G$. 

A set $D \subset G\times \overset{m}{\ldots} \times G = G^m$ is called \emph{definable by positive existential formulas} in $(G, \mc{R})$, or \emph{PE-definable} in $(G, \mc{R})$, if there exist finitely many (finite) systems $\mc{E}_{D,1}, \dots, \mc{E}_{D, r}$ of equations with constraints in $(G, \mc{R})$, variables $(x_1, \dots, x_m, y_1, \dots, y_k)=(\mathbf{x}, \mathbf{y})$, and constants from $G$, satisfying the following: 
\begin{itemize}
\item a tuple $\mathbf{a}\in G^m$ is in $D$
 if and only if  there exists $i\in \{1, \dots,r\}$ such that the system $\mc{E}_{D,i}(\mathbf{a}, \mathbf{y})$ on variables $\mathbf {y}$ has a solution in $G$. 
 \end{itemize}
  In this case $\mc{E}_{D,i}$ is said to \emph{PE-define} $D$ in $G$.

If $\mc{R}=\emptyset$ then we simply say that $D$ is \emph{PE-definable} in $G$.
\end{definition}

\begin{example}\label{ex:central}
A typical PE-definable set in a monoid or a group $G$ is the centraliser of an element $g$, that is, the set $C_G(g) = \{h\in G\mid gh=hg\}$; an element $x\in G$ belongs to $C_G(g)$ if and only if $x$ is a solution to the equation $xg=gx$.

Another classical example is the centre $Z(G)=\{g\in G\mid gh=hg \ \forall h\in G\}$ of a group $G$ generated by a finite set $\{g_1, \dots, g_n\}$. Indeed, an element $x \in Z(G)$ if and only if $x$ satisfies the system of equations $[x,g_1]=1, \dots, [x,g_n]=1$.
\end{example}

The use of disjunctions of systems allows us to work comfortably with unions of PE-definable sets: if $S_1$ and $S_2$ are PE-defined in $G$  and $\Phi_1, \Phi_2$ are two disjunctions of systems of equations PE-defining these sets, then $\Phi_1\vee \Phi_2$ is again a PE-formula, and it PE-defines $S_1\cup S_2$.

\begin{definition}\label{d: e-int}
Let $G$ and $H$ be  groups and let $\mc{R}_G$ and $\mc{R}_H$ be  sets of relations in $G$ and $H$, respectively. We say that $(H,\mc{R}_H)$   is  \emph{PE-interpretable} in  $(G, \mc{R}_G)$ if there exists $n\in \mathbb{N}$, a subset $D \subseteq G^n$ and an onto map  (called the \emph{interpreting} map)
$$
\phi: D \twoheadrightarrow H,
$$ 
satisfying the following conditions.
\begin{enumerate}
\item $D$ is PE-definable in $(G,\mc{R}_G)$. 
\item  The preimage under $\phi$ of the operation $\cdot_H$, that is, the set $$\{(\phi^{-1}(x_1), \phi^{-1}(x_2), \phi^{-1}(x_3)) \mid x_1\cdot_H x_2 = x_3, \ x_1, x_2,x_3\in H\},$$ is PE-definable in $(G, \mc{R}_G)$.

In this case we say that the operation $\cdot_H$ is PE-definable in $(G,\mc{R}_G)$. 
%
%
\item The preimage under $\phi$ of the graph of the equality relation in $H$, namely the set $\{(\phi^{-1}(x_1), \phi^{-1}(x_2))  \mid x_1=x_2,\ x_1, x_2\in H \})$,  is PE-definable in $(G,\mc{R}_G)$.
\item For each relation $r\in\mc{R}_H$ with arity $t_r$, the preimage under $\phi$ of the graph of $r$, namely the set 
$\{(\phi^{-1}(x_1), \dots, \phi^{-1}(x_{t_r}))\mid x_1,\ldots, x_{t_r}\in H  \text{ satisfy } r\}$, is PE-definable in $(G,\mc{R}_G)$.

In this case we say that the relation $r$ is PE-definable in $(G,\mc{R}_G)$.
\end{enumerate}
In this paper all PE-interpretations have $n=1$.

One can establish the analogous notion of \emph{a ring $R$ which is PE-interpretable in $(G, \mc{R})$}. The definition in this case is analogous, requiring the preimages of the graphs of the ring sum, multiplication, and the equality relation, to be PE-definable in $(G, \mc{R})$.

If $\mc{R}_G=\emptyset$ then we simply speak of PE-interpretability in $G$.
\end{definition}

\begin{proposition}\label{prop:DPred} (Reduction of Diophantine problems)\label{RedCor}
Let $G$ be a group and let $\mc{R}_G$ be a set of relations in $G$. Let $M$ be either a ring or a group, and let $\mc{R}_M$ be a set of relations in $M$.

If the structure $(M, \mc{R}_M)$ is PE-interpretable in $(G,\mc{R}_G)$, then $\mc{DP}(\mc{M}, \mc{R}_M)$ is reducible to $\mc{DP}(G,\mc{R}_G)$. Consequently, if $\mc{DP}(\mathcal{M}, \mc{R}_M)$  is undecidable, then $\mc{DP}(G,\mc{R}_G)$ is undecidable as well.
%
\end{proposition}


The result of Proposition \ref{prop:DPred} is a direct consequence of the following lemma, which, roughly, states that when a structure $M_1$ is PE-interpretable in a second structure $M_2$, then one can ``rewrite'' any disjunctions of systems of equations with relations (or constraints) in $M_1$ as an equivalent disjunction of systems of equations with relations in $M_2$. Lemma \ref{equation_reduction} is a variation of well-known results from model theory, and the proof presented here essentially follows step-by-step the proofs of such well-known results, such as Theorem 5.3.2 in \cite{Hodges}. The proof can be easily adapted to the case that we are PE-interpreting a ring $R$ into a group enriched with relations.
\begin{lemma}\label{equation_reduction}
 Let $(G,\mc{R}_G)$ and $(H,\mc{R}_H)$ be two groups enriched with relations $\mc{R}_G, \mc{R}_H$, respectively. Let $\phi$  be an interpreting map of $(H, \mc{R}_H)$ into $(G,\mc{R}_G)$,
with $\phi: D \subseteq G^n \to H$ (see Definition \ref{d: e-int}). Let $\sigma(\mathbf{x})=\sigma(x_1, \dots, x_n)$ be an arbitrary disjunction of systems of equations with constraints in $(H, \mc{R}_H)$. 

Then there exists a disjunction of systems of equations with constraints $\Sigma_\sigma(\mathbf{y})$ in $(G, \mc{R}_G)$, such that a tuple $(b_1, \dots, b_n)\in G^n$ is a solution to  $\Sigma_\sigma(y_1,\ldots,y_n)$ if and only if $(b_1, \dots, b_n)\in D$ and $\phi(b_1, \dots, b_n)$ is  a solution to $\sigma$.
\end{lemma}

\begin{proof} (of Lemma \ref{equation_reduction})
We can assume without loss of generality, and we do so, that $\sigma$ consists of a single system of equations with  constraints. 

 We claim that, by introducing new variables and new equations, we can rewrite $\sigma$ so that $\sigma$ consists of equations $\sigma_1, \dots, \sigma_m$ with the following property:  For all $i=1, \dots, m$, $\sigma_i$ is either of the form $z=xy$, $x=y$, $x=h$, or $r(x_1, \dots, x_{n_r})$ for some variables $x,y,z, x_1, \dots, x_{n_r}$, some constant element $h\in H$, and some relation $r\in\mc{R}_H$ of arity $n_r$.  The lemma follows from the claim, since by the definition of PE-interpretability, the present lemma is true for each of $\sigma_i$, $i=1, \dots, m$. Hence, it suffices to take $\Sigma_\sigma$ to  be  $\Sigma_{\sigma_1}\wedge\dots\wedge\Sigma_{\sigma_n}$.

 We now prove the claim. We proceed by induction on the syntactic length (i.e.\ the number of symbols) $|\sigma|$ of $\sigma$, the base cases being immediate by, again, the definition of PE-interpretability. It suffices to prove the claim in the case that $\sigma$ consists of a single equation or relation (i.e.\ $m=1$), since once this is proved, if there are $> 1$ equations or relations, we can apply this case separately for each equation or relation.
 
 First of all, assume that $\sigma$ consists of a single equation, not of one of the desired short forms.  Since we are working over groups (or rings), we can assume that $\sigma$ has the form $w(x_1, \dots, x_n)=1$ for some word $w$ with constants and variables $x_1,\dots, x_{n}$.  We will omit the references to the variables $x_1,\dots, x_n$ to make the presentation more readable. Let $t_1,t_2$ be the first two symbols of $w$, and let $w'$ be a word such that $w=t_1 t_2 w'$. We can rewrite $\sigma$ into the equivalent system of equations $t_1 t_2 y=1 \wedge y^{-1}w'=1$, where $y$ is a new variable. The syntactic length of each one of these equations is strictly less than $|\sigma|$, and then we can proceed by induction.   
 
 If $\sigma$ consists of a single relation, then it has the form $r(w_1(\mathbf{x}_1), \dots, w_{n_r}(\mathbf{x}_{n_r})))$ for some relation $r\in\mc{R}_H$ of arity $n_r$ and some tuples of variables $\mathbf{x}_{1}, \dots, \mathbf{x}_{n_r}$ (that is, each $\mathbf{x}_{i}$ is a tuple of variables). In this case one can proceed similarly by rewriting the word $w_1$ in an analogous way as the word $w$ was rewritten in the previous case, obtaining an equivalent conjunction of equations and relations of the form  $r(t_1t_2y, w_2, \dots, w_{n_r}) \wedge y^{-1}w_1'=1$, where $t_1$ and $t_2$ are the first two symbols of $w_1$ (again, we have omitted writing down the variables). Each one of the two terms in this conjunction has syntactic length smaller than $|\sigma|$, and again we can proceed by induction. This proves the claim.
\end{proof}

Proposition \ref{prop:DPred} is well-known if one replaces disjunctions of equations by first-order formulas, and the problems  $\mc{DP}(M, \mc{R}_M)$ and $\mc{DP}(G,\mc{R}_G)$ by the elementary theory of $M$ enriched with relations $\mc{R}_M$ and $G$ enriched with relations $\mc{R}_G$ (see Theorem 5.3.2 of \cite{Hodges} and its consequences).  Replacing, in \cite[Theorem 5.3.2]{Hodges}, first-order formulas by finite disjunctions of equations with relations, and elementary theories by the appropriate Diophantine problems with relations, gives Proposition \ref{prop:DPred}.

\section{Results for free groups and free monoids}\label{sec:free}

The results in this section will be generalised later in the paper and might be known to the experts, but the simpler arguments for free groups are worth including here. 
We use the ideas in \cite{RichardBuchi1988} to encode addition and multiplication in $\mbb{Z}$ into the $\mc{DP}(F(\Sigma), \ab)$, and by the undecidability of Hilbert's 10th, get the same result for free groups. The starting point is the result for free monoids below.

\begin{theorem}[\cite{RichardBuchi1988}, Cor. 4]\label{Buchi1}
The Diophantine problems $\mc{DP}(\Sigma^*, \ab)$ and $\mc{DP}(\Sigma^*, \{|| . ||_i\}_i)$ are undecidable.
\end{theorem}


Lemma \ref{lemma: DP_implications} shows that there are connections between the Diophantine Problems with different kinds of constraints, when working with free monoids or free groups. These connections can be established by interpretability. 

\begin{lemma}\label{lemma: DP_implications}
Let $\Sigma=\{s_i \mid 1 \leq i \leq m\}$ be the set of generators of $\Sigma^*$ or $F(\Sigma)$.
\leavevmode
\begin{enumerate}	
\item[(a)] The $\mc{DP}(F(\Sigma),\{| . |_i\}_i$) is decidable if and only if the $\mc{DP}(F(\Sigma),\ab$) is decidable.

\item[(b)] Similarly, $\mc{DP}(\Sigma^*,\{|| . ||_i\}_i$) is decidable if and only if $\mc{DP}(\Sigma^*,\ab$) is decidable, and the decidability of these problems implies the decidability of $\mc{DP}(\Sigma^*, \Length$).
\end{enumerate}
\end{lemma}

\begin{proof}

(a) The decidability of $\mc{DP}(F(\Sigma),\{| . |_i\})$ clearly implies the same for $\mc{DP}(F(\Sigma),\ab)$. 

Conversely,  we prove that each $|.|_i$ can be interpreted by a system of equations in ($F(\Sigma),\ab$). Indeed, two elements $x, y$ satisfy $|x|_i=|y|_i$ if and only if there exist elements $x', y' \in \langle s_1 \rangle \cdots \langle s_{i-1} \rangle \cdot  \langle s_{i+1} \rangle \cdots \langle s_n \rangle$ and $w \in \langle s_i \rangle$ such that $\ab(x)=\ab(wx') \ \wedge \ab(y)=\ab(wy')$. Since the centraliser of any generator $s_k$ in a free group is $\langle s_k \rangle$ (infinite cyclic and generated by $s_k$), we can express the above equivalently as:  $|x|_i=|y|_i$ if and only if there exist 
$z_1, \dots, z_{i-1}, z_{i+1}, \dots, z_n, u_1, \dots, u_{i-1}, u_{i+1}, \dots, u_n, w,  x', y'$ such that \begin{align*} &\bigwedge_{j\neq i}[z_j,s_j]=1 \  \wedge \ \bigwedge_{j\neq i}[u_j,s_j]=1 \ \wedge \ [w,s_i]=1 \ \wedge \ \\  &x'= \prod_{j\neq i} z_i \ \wedge \ y'= \prod_{j\neq i} u_i \ \wedge \ab(x)=\ab(wx') \ \wedge \ab(y)=\ab(wy').\end{align*}

Hence $|\cdot|_i$ is PE-interpretable in $(F(\Sigma), \ab)$ for all $i=1, \dots, n$. It follows that $\mc{DP}(F(\Sigma), \{|.|_i\}_i)$ is reducible to $\mc{DP}(F(\Sigma), \ab)$.

(b) For free monoids this is proved in \cite{RichardBuchi1988}. 
\end{proof}

The decidability of $\mc{DP}(\Sigma^*, \Length)$ in Lemma \ref{lemma: DP_implications}(b) is a major open problem in theoretical computer science (see \cite{DayManeaWE}). However, \ref{lemma: DP_implications} does not have any consequences regarding $\mc{DP}(\Sigma^*, \Length)$ since $\mc{DP}(\Sigma^*, \ab)$ and $\mc{DP}(\Sigma^*, \{|| . ||_i\}_i)$ are undecidable.


\begin{theorem}\label{thm:free_gp}
The $\mc{DP}(F(\Sigma), \ab)$ and $\mc{DP}(F(\Sigma),\{| . |_i\}_i)$ are undecidable if $|\Sigma| \geq 2$.
\end{theorem}

\begin{proof} 
As before, let $\Sigma=\{s_i \mid 1 \leq i \leq m\}$ be the set of generators of $F(\Sigma)$.

 By Lemma \ref{lemma: DP_implications}, it suffices to prove that $\mc{DP}(F(\Sigma), \{|.|_i\}_i)$ is undecidable. To prove this we interpret the ring $(\mbb{Z}, \oplus, \odot)\cong (\{s^t \mid t\in \mbb{Z}  \}, \oplus, \odot)$ into $(F(\Sigma), \{|.|_i\}_i)$, where $s=s_1\in \Sigma$. We define $s^{t_1} \oplus s^{t_2} =_{def} s^{t_1 + t_2}$ and $s^{t_1} \odot s^{t_2} =_{def} s^{t_1  t_2}$. We need to establish the sets, maps and conditions of Defintion \ref{d: e-int}: to start with, the interpreting map $\phi$ is the natural bijection between $D=\{s^t \mid t\in \mbb{Z} \}$ and $\mbb{Z}$, and $n=1$.

The set $D$ is defined by the equation $[x, s]=1$, the operation $\oplus$ coincides with the group multiplication in $F(\Sigma)$, and the equality relation on $D$  is also trivially PE-definable.  Hence the abelian group $(\{s^t \mid t\in \mbb{Z}\}, \oplus)$ is PE-interpretable in $F(\Sigma)$. 

We now PE-interpret $\odot$, that is, we show that the set $\{(s^{t_1}, s^{t_2}, s^{t_3}) \mid t_1t_2=t_3, t_i \in \mbb{Z}\}$ is PE-definable in $(F(\Sigma), \{|.|_i\}_i)$. Write $a_i=s^{t_i}(=s_1^{t_i})$ for $t_i \in \mbb{Z}$ ($i=1,2,3$). We claim that $a_1, a_2, a_3 \in D$ satisfy $a_1\odot a_2 = a_3$ if and only if there exist $b, c\in F(\Sigma)$ such that \begin{equation}\label{def_multiplication}[b, s_2]= 1 \wedge [c, s_1b]=1 \wedge |a_1|_{s_1} = |b|_{s_2} \wedge |a_2|_{s_1}=|c|_{s_1} \wedge |a_3|_{s_1}= |c|_{s_2}.\end{equation}
Once the claim is proved, we obtain by definition that (the graph of) $\odot$ is PE-definable in  $(F(\Sigma), \{|.|_i\}_i)$. 

We prove first the converse of the claim. Let $b,c$ be elements satisfying the above equations with exponent-sum constraints. Then $b= s_2^{t_1}$ and $c= (s_1s_2^{t_1})^{t_2}$, and the last equality implies that $t_3 = |(s_1s_2^{t_1})^{t_2}|_{s_2} = t_1t_2$.  Hence $a_1\odot a_2=a_3$.

For the direct implication, suppose   $a_1\odot a_2=a_3$, so $t_3=t_1t_2$, and take $b = s_2^{t_1}, c= (s_1 b)^{t_2}$. Then $b,c$ satisfy \eqref{def_multiplication}.
\end{proof}

\begin{remark}\label{rem:free_gp} We make two observations concerning the proof of Lemma 
\ref{lemma: DP_implications} that will be relevant later, especially for what we call the Technical Lemma \ref{l: technical_lemma}.

First, note that the interpretability of integer multiplication $\odot$ in the proof above is possible because the set $\{(s_1s_2^{t_1})^{t_2} \mid t_1, t_2 \in\mbb{Z}\}$ is PE-definable in the free group $F(\Sigma)$. 

Second, in the proof above we  interpreted the ring $(\mbb{Z}, \oplus, \odot)$ in $(F(\Sigma), \{|.|_i\}_i)$. However, in the proof of Lemma 
\ref{lemma: DP_implications} we show that $\{|.|_i\}_i$ is PE-interpretable in $(F(\Sigma), ab)$, so we can also interpret the ring $(\mbb{Z}, \oplus, \odot)$ in $(F(\Sigma), ab)$. \end{remark}

\section{Groups with finite abelianisation}\label{sec:fin_ab}
Let $G^{\ab}=G/G'$ be the abelianisation of $G$. In this section we show that solving the $\mc{DP}(G,\ab)$ can be reduced, when $G^{\ab}$ is finite, to solving the Diophantine Problem in $G$ with recognisable constraints (see Definition \ref{def:rat_reg_con}). We then apply this to important classes of groups where $\mc{DP}$ with recognisable constraints is known to be decidable, such as hyperbolic groups with finite abelianisation and graph products of finite groups.


\begin{remark}\label{rmk:recognisable}
By \cite[Proposition 6.3]{herbst_thomas} a subset of a group
$G$ is recognisable if and only if it is a union of cosets of a
subgroup of finite index in $G$, and hence a union of cosets of a normal
subgroup of finite index (the core of a finite index subgroup will be both
normal in $G$ and of finite index in $G$); recognisability 
 is independent of the choice of generating set for $G$.
\end{remark}


\begin{proposition}\label{thm:finiteab}
Let $G$ be a group with finite abelianisation. Then $\mc{DP}(G, \ab)$ is reducible to the Diophantine Problem in $G$ with recognisable constraints.
\end{proposition}
\begin{proof}  Let $\ab:G\to G^{\ab}$ be the natural abelian projection
to $G^{\ab}=G/G'$, which is finite; this implies that $G'$ has finite index, and so any coset of $G'$ is recognisable by Remark \ref{rmk:recognisable}. Let $\mathcal E$ be a system of equations on variables $\mathbf{x}=\{X_1, \dots, X_k\}$ in $G$ with abelian constraints $\mathcal A$ consisting of a system of equations over $\mathbf{x}$ in $G^{\ab}$. 

We claim that any equation in $\mathcal A$ can be seen as a recognisable constraint imposed on $\mathcal E$.
Because of the commutativity of $G^{\ab}$, any equation in $\mathcal A$ can be rewritten as $W(\mathbf{x})=\alpha$, where $W(\mathbf{x})$ is a constant-free word on the variables $\mathbf{x}$ and $\alpha \in G^{\ab}$.
We introduce a new variable $Z$ for each such $W(\mathbf{x})$, and set $Z=\alpha$ (which implies $W(\mathbf{x})=Z$), so that we can more easily apply constraints to $Z$ instead of $W(\mathbf{x})$ below; let $\mathbf{x}'$ be the expanded set of variables for $\mathcal E'$, which includes the new $Z$ variables, and consists of $\mathcal E$ together with the newly introduced equations $Z=\alpha$. Then $W(\mathbf{x})=\alpha$ holds in $G^{\ab}$ if and only if $Z\in \ab^{-1}(\alpha)$. This is equivalent to $Z\in \alpha G'$, that is, $Z$ belongs to a coset of the commutator subgroup $G'$. Since any coset of $G'$ is recognisable, all equations in $\mathcal A$ can be seen as recognisable constrains imposed on the set of variables $\mathbf{x}'$ (more precisely, 
on the variables in $\mathbf{x}'$ that do not belong to $\mathbf{x}$).
%
\end{proof}

Hyperbolic groups have been studied extensively in geometric group theory, and the Diophantine Problem is known to be decidable by work of Rips \& Sela \cite{RS95} in the torsion-free case and Dahmani \& Guirardel \cite{DG} in the general case. We do not give the definitions and background on hyperbolic groups here, but refer the reader to \cite{MSRInotes}.

\begin{theorem}\label{thm:hyp_finiteab}
Let $G$ be a hyperbolic group with finite abelianisation. Then $\mc{DP}(G, \ab)$ is decidable.
\end{theorem}

For the following proof we will need to recall an important concept from \cite{DG}.

\begin{definition}\label{def:rat_con} Let $S$ be a finite symmetric generating set for $G$.
\ \begin{enumerate}
\item[(1)] A regular subset $L'$ of $S^*$ is
\emph{quasi-isometrically embedded} (q.i. embedded) in $G$ if there exist
$\lambda \geq 1$ and $\mu \geq 0$ such that, for any $w \in L'$,
$|\pi(w)|_G \geq \frac{1}{\lambda}|w|-\mu.$
\item[(2)] A rational subset $L$ of $G$ is \emph{quasi-isometrically
embeddable} (q.i. embeddable) in $G$ if there exists a quasi-isometrically embedded
regular subset  $L'$ of $S^*$ such that $\pi(L')=L.$
\end{enumerate}
\end{definition}

\begin{proof} (of Theorem \ref{thm:hyp_finiteab} )
The Diophantine Problem with quasi-isometrically embeddable rational constraints is decidable in hyperbolic groups by \cite{DG}. 

We will show that a recognisable set $L\subset G$ can be made to quasi-isometrically embed: let $QGeo(G,\mu,\lambda)$ be the set of $(\lambda,\mu)$-quasi-geodedics over $S$, which is known to be regular, and consider the set $P=\pi^{-1}(L)\cap QGeo(G,\mu,\lambda)$. Then $P$ is a regular set, as the intersection of two regular sets, and $\pi(P)=L$.
Since the abelianisation $G^{\ab}$ of $G$ is finite, the theorem follows from Proposition \ref{thm:finiteab} and \cite{DG}.
\end{proof}

%
%
%
Graph products of groups are another widely studied class of groups, and the graph product construction generalises both direct and free products. As in the case of RAAGs, we start with a finite undirected graph $\Gamma$ with no loops at any vertex and no multiple edges between two vertices. We associate a group to each vertex of $\Gamma$, and let the graph product $\mc{G}\Gamma$ be the group whose generator set is the union of the vertex group generators, and the relations consist of the relations in the vertex groups together with commuting relations between the generators of any vertex groups connected by an edge in $\Gamma$. RAAGs are simply graph products where each vertex group is equal to $\mbb{Z}$.

\begin{example}\label{ex:graph_product}
Let $\Gamma$ be the graph below with vertex groups $G_1, G_2, G_3, G_4$, where $G_i$ is the cyclic group of order $2i$ generated by $x_i$. The graph product $\mc{G}\Gamma$ based on $\Gamma$ has the presentation $\langle x_1,x_2,x_3,x_4 \mid [x_1,x_2]=[x_2,x_3]=[x_3,x_4]=1, x_1^2=x_2^4=x_3^6=x_4^8=1 \rangle$; the abelianisation of $\mc{G}\Gamma$ is the direct product $C_2 \times C_4 \times C_6 \times C_8$ of finite cyclic groups of orders $2, 4, 6, 8$, respectively.
 \[
  \xymatrix{\stackrel{G_1}{\bullet}\ar@{-}[r] &\stackrel{G_2}{\bullet} \ar@{-}[r]&\stackrel{G_3}{\bullet}\ar@{-}[r] &\stackrel{G_4}{\bullet} }.
 \]

 \end{example}

\begin{theorem}\label{thm:graph_prod}
Let $G$ be a graph product of finite groups. Then $\mc{DP}(G, \ab)$ is decidable.
\end{theorem}

\begin{proof}

Equations with recognisable constraints in graph products of finite groups are decidable by Diekert and Lohrey's work \cite{DLijac}, and since a graph product of finite groups has finite abelianisation, the theorem follows from Proposition \ref{thm:finiteab}.
\end{proof}

The following is a consequence of Theorem \ref{thm:graph_prod}, since right-angled Coxeter groups are graph products where each vertex group is equal to the cyclic group of order two $C_2$.

\begin{cor}\label{thm:RACG}
Let $G$ be a right-angled Coxeter group. Then $\mc{DP}(G, \ab)$ is decidable.
\end{cor}

\section{Interpreting the ring $\mathbb{Z}$ via the Diophantine Problem with abelian constraints}\label{sec:technical}

 In this section we present a general technical result (Lemma \ref{l: technical_lemma}) that will allow us to obtain undecidability of the $\mc{DP}(G,\ab)$ for all non-abelian RAAGs and hyperbolic groups with `large' abelianisation, that is, $b_1(G)\geq 2$. 

\begin{definition}\label{d: R1_R2}
		Let $G$ be a group generated by a finite set $S$.
		\begin{enumerate}
		
		\item[(1)] If $X\subseteq S$ is a subset of \ap \ elements, let $R_X=R_X(\cdot,\cdot)$ be the binary relation defined as: $g, h\in G$ satisfy $R_X$ if and only if $|g|_x=|h|_x$ for all $x\in X$.
	
	We will simply write $R_s$ if $X=\{s\}$ consists of a single element.
	\item[(2)] For any distinct \ap \ elements $s,t \in S$ let $R_{s,t}=R_{s,t}(\cdot,\cdot)$ be the binary relation defined as: $g,h\in G$ satisfy $R_{s,t}$ if and only if $|g|_s = |h|_t$. 
	\end{enumerate}
	
\end{definition}

Recall that an $n$-ary relation $R$ in $G$ is \emph{PE-definable in $G$} if the set $$\{(g_1, \dots, g_n)\in G^n \mid (g_1, \dots, g_n) \text{ satisfies } R\}$$ is PE-definable in $G$.

We begin with the following basic observation.

\begin{lemma}\label{lem:Ks_def}
    Let $G$ be a group generated by a set $S$, let $\{s_1, \dots, s_t\}\subseteq S$ be a set of \ap \ generators, and let $$K_{s_1, \dots, s_t}:=\{g \in G \mid |g|_{s_i}=0  \ \forall i=1,\dots, t\}.$$ Then $K_{s_1, \dots, s_t}$ is a normal subgroup of $G$, and for each $g\in G$ there exists $k_g\in K_{s_1, \dots, s_t}$ such that
 $g = k_g \Pi_{i=1}^t s_i^{|g|_{s_i}}.$
\end{lemma}
\begin{proof} The normality of $K_{s_1, \dots, s_t}$ follows from the fact that $|\cdot|_{s_i}$ is a homomorphism (Lemma \ref{l: morphism}) with kernel $K_{s_i}$, and $K_{s_1, \dots, s_t}=\bigcap_{i=1 \dots t} K_{s_i}$. The last part of the lemma follows once we note that  $G' \leq K_{s_1, \dots, s_t}$, which implies $G/K_{s_1, \dots, s_t}$ is abelian; thus, modulo $K_{s_1, \dots, s_t}$, the $s_i$ letters can be shuffled and placed in any order required, and so $g$ can be written as $k_g \Pi_{i=1}^t s_i^{|g|_{s_i}}$, where $k_g\in K_{s_1, \dots, s_t}$.
%
\end{proof}

\newcommand{\Phii}[1]{|{#1}|_{s_1}}

The following lemma is an abstraction of the arguments used in the proof of Theorem \ref{thm:free_gp} about the undecidability of $\mc{DP}(F(S), \ab)$, and its complicated form is required for the proofs involving RAAGs. In the case of hyperbolic groups the conjugating elements $h_g$ from Lemma \ref{l: technical_lemma} (\ref{def_set}) do not occur, and for free groups this has an even simpler form. Remark \ref{rem:free_gp} suggests how the free group is a special case of the lemma below: Item (1) is satisfied because all $|\cdot|_i$ are PE-interpretable in $(F(\Sigma), ab)$, so the relations $R_{s_1}$, $R_{s_2}$, and $R_{s_1, s_2}$ are PE-definable in $(F(\Sigma), ab)$; for Item (2) we can take $h_g$ to be trivial, and we do not need to include $K_{s_1,s_2}$ in the set described by (\ref{def_set}). 

\begin{lemma}[General technical lemma]\label{l: technical_lemma}
	 Let $G$ be a group generated by a set $S$. Let $s_1, s_2 \in S$ be \ap \ elements with $s_1\neq s_2$. 	 Suppose that a group $G$ satisfies (1) and (2) (which is equivalent to (2')). Then the problem $\mc{DP}(G, \ab)$ is undecidable.
\begin{itemize}	 
	\item[(1)] The relations $R_{s_1}$, $R_{s_2}$, and $R_{s_1, s_2}$ are PE-definable in $(G, \ab)$.

	 \item[(2)] There exists a disjunction $\bigvee_{i=1}^{m}\Sigma_i(z, x, y_1, \dots, y_n)$ of systems $\Sigma_i$ of equations 
	 on variables $z,x,y_1, \dots, y_n$, such that for each $g\in K_{s_1}$ there exists $h_g\in G$ satisfying:
	  the disjunction of systems $ \bigvee_{i=1}^{m}\Sigma_i(g, x, y_1, \dots, y_n)$ PE-defines  in $(G, \ab)$ the set 
	 \begin{equation} \label{def_set}
	 h_g^{-1}\{(s_1  s_2^{|g|_{s_2}})^t\mid t\in \mbb{Z}\} K_{s_1, s_2} h_g \subseteq G.
	 \end{equation}
	 \item[(2')]  By definition, (2) is equivalent to: a tuple $(x', y_1', \dots, y_n')$ is a solution to some of the systems of equations $\Sigma_i(g, x, y_1, \dots, y_n)$ ($i=1,\dots, m$) if and only if $$x'\in h_g^{-1}\{(s_1  s_2^{|g|_{s_2}})^t\mid t\in \mbb{Z}\} K_{s_1, s_2} h_g.$$  

	 
\end{itemize}

		
	\end{lemma}
	
\begin{proof}
We will PE-interpret the ring $(\mbb{Z}, +, \cdot)$ in $(G,\ab)$. 
 As interpretation map we use $|\cdot|_{s_1}: G \to \mbb{Z}$. The preimage under this map of the equality relation in $\mbb{Z}$, i.e.\ of the set $$\{(g,h)\in G \times G\mid |g|_{s_1}=|h|_{s_1}\},$$ is PE-definable in $(G,\ab)$, because the relation $R_{s_1}$ is by (1).

We now PE-define in $G$ the preimage by $|\cdot|_{s_1}$ of integer addition, namely the set $$S_+=\{(g_1, g_2, g_3) \in G^3\mid \quad \Phii{g_1} +\Phii{g_2} = \Phii{g_3}\}.$$ We claim that a tuple $(g_1, g_2, g_3) \in G^3$ belongs to $S_+$ if and only if $g_3K_{s_1} = g_1g_2K_{s_1}$. Indeed, if $g_3K_{s_1} = g_1g_2K_{s_1}$, then $g_3 = g_1g_2k$ for some $k\in K_{s_1}$, so $\Phii{g_3} = |g_1g_2k|_{s_1} = |g_1|_{s_1} + |g_2|_{s_1},$ where for the second equality we used Lemma \ref{l: morphism}. Conversely, if $\Phii{g_1} +\Phii{g_2} = \Phii{g_3}$ then $|g_1|_{s_1} + |g_2|_{s_2} = |g_1g_2|_{s_1} = |g_3|_{s_1},$ hence $ |g_1g_2g_3^{-1}|_{s_1}= 0$ and so $g_1g_2g_3^{-1}\in K_{s_1}$, and the claim follows. By the claim, we have $S_+ = \{(g_1, g_2, g_3) \in G^3\mid g_1g_2g_3^{-1}\in K_{s_1}\}$. It now follows from the fact that $K_{s_1}$ is PE-definable in $G$ that $S_+$ is  also PE-definable $K_{s_1}$ in $G$.

We next PE-define in $(G,\ab)$ the preimage under $|\cdot|_{s_1}$ of integer multiplication, i.e. of the set $$S_{\odot} = \{(g_1, g_2, g_3) \in G^3\mid \quad \Phii{g_1} \cdot \Phii{g_2} = \Phii{g_3}\}.$$ 

We claim that three elements $a_i \in G, i=1,2,3$, belong to $S_{\odot}$ if and only if there exist $b, c\in G$ such that \begin{align}\label{def_multiplication_1} &\left(G\vDash \exists y_1 \dots \exists y_n, \bigvee_{i=1}^m\Sigma_i(b, c, y_1, \dots, y_n) \right) \wedge \\ \label{def_multiplication_2}\wedge &\left( |a_1|_{s_1} = |b|_{s_2}\right) \wedge  \left(|b|_{s_1}= 0 \right)\wedge \left( |a_2|_{s_1}=|c|_{s_1}\right) \wedge  \left(|a_3|_{s_1}= |c|_{s_2}\right).\end{align}
If we let $t_i:=|a_i|_{s_1}$, where $t_i$ ($i=1,2,3$) are integers, note that by the definition of $S_{\odot}$ we have $a_1, a_2,a_3 \in S_{\odot}$ if and only if $t_3 = t_1t_2$.

We first prove the converse of the claim. Let $b,c$ be elements satisfying (\ref{def_multiplication_1}) and (\ref{def_multiplication_2}). Then $|b|_{s_2}= t_1$ and $|b|_{s_1}=0$. 
Due to the assumption (2) in the lemma, we have $$c=h_g^{-1} (s_1 s_2^{|b|_{s_2}})^{t} k  h_g = h_g^{-1}(s_1 s_2^{t_1})^{t}k h_g$$ for some $t \in \mbb{Z}$ and $k\in K_{s_1, s_2}$ since $c$ is part of a solution to some of the $\Sigma_i$ systems. By the additivity of $|\cdot|_{s_1}$ (Lemma \ref{l: morphism}) we get that $|c|_{s_1}=t$, and from $|c|_{s_1}=|a_2|_{s_1}=t_2$ we have $t=t_2$.  Moreover, again by additivity of $|\cdot|_{s_2}$, we obtain $|c|_{s_2} = t_1t= t_1 t_2$. The equality  $|c|_{s_2} = |a_3|_{s_1} = t_3$ in (\ref{def_multiplication_2}) now implies that $t_3 = t_1t_2$, as required.

 To prove the direct implication of the claim, suppose   $a_1, a_2, a_3 \in S_{\odot}$, so  $t_3=t_1t_2$. Take $b = s_2^{t_1}, c= h_b^{-1}(s_1 b)^{t_2}h_b$ (where $h_b$ is the element $h_g$ given in Item 2 of the statement of the lemma when taking $g=b$). It is straightforward to verify that these elements satisfy the conditions from \eqref{def_multiplication_2}. To check that $b,c$ satisfy \eqref{def_multiplication_1}, we must show that there exist  $y_1, \dots, y_n \in G$ that form a solution to one of the  systems of equations $\Sigma_i(b,c, y_1, \dots, y_n)$ ($i=1,\dots, m$). By assumption, a tuple $(x', y_1', \dots, y_n')$ is a solution to some of the systems $\Sigma_i(b,x, y_1, \dots, y_n)$ ($i=1,\dots, m$) if and only if  $$x'\in h_b^{-1}\{(s_1  s_2^{|b|_{s_2}})^t\mid t\in \mbb{Z}\}K_{s_1, s_2} h_b.$$
 We will argue that for any tuple of elements $(y_1, \dots, y_n)\in G$ we have that $y_1, \dots, y_n$ is a solution to one of the systems  $\Sigma_i(b,c, y_1, \dots, y_n)$ ($i=1,\dots, m$). Indeed, we have $$c = h_b^{-1}(s_1 s_2^{|b|_{s_2}})^{t_2} h_b\in h_b^{-1}\{(s_1  s_2^{|b|_{s_2}})^t\mid t\in \mbb{Z}\}K_{s_1, s_2}h_b,$$ hence $(c, y_1, \cdots, y_n)$ is a solution to some of the systems $\Sigma_i(b,x, y_1, \dots, y_n)$ ($i=1, \dots, m$) for any tuple of elements $y_1,\dots, y_n$ 
 This completes the proof of the claim. 

The proof of the lemma now follows from the fact that all conditions  appearing in  \eqref{def_multiplication_1} are PE-definable in $(G, \ab)$ by the assumptions (1) in the statement of the lemma.  
\end{proof}


\section{Partially commutative groups (RAAGs)}\label{sec:RAAGs}

We now show that the Diophantine Problem with abelianisation constraints is undecidable for RAAGs that are not abelian. The idea of the proof is, like in the free group case, to find two sufficiently independent elements which can be manipulated so that we can encode the ring $(\mbb{Z}, \oplus, \odot)$ in a RAAG. Finding two such elements is more difficult than for free groups: we first need to find two sets of vertices/generators, which we call \emph{weak modules}, and then  obtain the two abelian-primitive elements by multiplying together all the elements in each weak module, respectively.

\subsection{Preliminaries on RAAGs}

Let $\Gamma$ be a finite, undirected graph with no auto-adjacent vertices, and let $G\Gamma$ be the partially commutative group or RAAG induced by $\Gamma$ (we refer to \cite{raags_intro} for a thorough introduction to RAAGs). We identify the vertices of $\Gamma$ with the corresponding generators of $G\Gamma$. Given a vertex $v\in V\Gamma$ we let $link(v)$ be the set of vertices adjacent to $v$, and we let $star(v) = \{v\}\cup link(v)$. This notation is extended to sets of vertices: if $S\subseteq V\Gamma$, then $link(S) = \bigcap_{v\in S}link(v)$ and $star(S)= \{S\}\cup link(S)$. A \emph{clique} is a set of vertices $S$ such that each vertex in $S$ is adjacent to all other vertices in $S$.  Given an element $g\in G\Gamma$, we let $supp(g)$, the \emph{support} of $g$ be the set of vertices $v$ such that  $g = g_1 v g_2$ for some $g_1, g_2\in G\Gamma$ and $g_1 v g_2$ is a geodesic, that is, cannot be shortened by any group relations or free cancellations. This set is well-defined. We define $link(g) $ and $star(g)$ as $link(supp(g))$ and $star(supp(g))$.
 
Given $v,u \in V\Gamma$ we write $v\leq u$ if and only if $star(v) \subseteq star(u)$. This constitutes a partial order on $V\Gamma$ which has been  frequently used in the study of RAAGs.

\begin{definition}\leavevmode
\begin{itemize}
\item[(1)] A pair of (not necessarily distinct) vertices $v,u\in V\Gamma$ is a \emph{weak pair} if both $v$ and $u$ are minimal with respect to $\leq$, and $star(v)=star(u)$. 
\item[(2)] Let $S$ be a non-empty set of minimal vertices in $\Gamma$ such that any two vertices in $S$ form a weak pair, and such that $S$ is maximal (with respect to set inclusion) among all sets of vertices with these properties. Then $S$ is  a \emph{weak  module}. 
\end{itemize}
\end{definition}

It is immediate to check the following.

\begin{lemma}\label{lem:weakmod}
Let $S$ be a week module and $v \in S$. Then
\begin{itemize}
\item[(i)] $S\subseteq star(v)= star(S)$,
\item[(ii)] $S$ is a clique, and
\item[(iii)] if $T$ is a weak module distinct from $S$, then $S\cap T=\emptyset$.
\end{itemize}
\end{lemma}


\begin{example}
In the graph $\Gamma_1$ we have $star(a) \subset star(b)$ and $star(d) \subset star(c)$, so $a\leq b$, $d \leq c$, the vertices $a$ and $d$ are minimal, and  there are two weak modules $S=\{a\}$, $T=\{d\}$, each consisting of a single vertex.
 \[
 \Gamma_1= \xymatrix{a\ar@{-}[r] &b \ar@{-}[r]&c\ar@{-}[r] &d}
 \]
In the graph $\Gamma_2$ there are two weak modules: $S=\{a, b\}$ and $T=\{d\}.$
 \[
 \Gamma_2=  \xymatrix{ a  \ar@{-}[d] \ar@{-}[r] & c\ar@{-}[ld] \ar@{-}[r] &d\\
  b}
 \]
\end{example}

\begin{remark}
Having weak modules with more than one vertex significantly increases the complexity of the arguments. The reader may assume that every weak module consists of a single minimal vertex, as in $\Gamma_1$ above. 
\end{remark}

The following describes centralisers $C_G(\cdot)$ in RAAGs
, first for vertices and then for general elements.

\begin{proposition}[\cite{Servatius1989AutomorphismsOG}]\label{p: centralizers_RAAGs}
	Let $G=G\Gamma$ be a RAAG and  let $v$ be a vertex  of $\Gamma$. Then $C_G(v)$ is the subgroup of $G$ generated by $star(v)$. More precisely, $C_G(v) = \langle v \rangle \times \langle link(v)\rangle$.
\end{proposition}
A \emph{cyclically reduced} element in a RAAG corresponds to a word $w$ that cannot be made shorter when cyclically permuting the letters of $w$ and then applying free cancellations and the group commuting relations.
In \cite{Servatius1989AutomorphismsOG} it is proved that for any cyclically reduced element $g$ in $G$ there exist unique elements $u_1, \dots, u_m$, which we call \emph{blocks}, and integers $n_1, \dots, n_m$ such that $[u_i,u_j]=1$ for all $i,j$ and $g=u_1^{n_1}\cdots u_m^{n_m}$. The latter expression is called the \emph{block-decomposition} of $g$.

\begin{theorem}[\cite{Servatius1989AutomorphismsOG}]\label{t: centralizers_RAAGS}
   Let $G=G\Gamma$ be a RAAG and  let $g\in G$. Let $h\in G$ be such that $g^h$ is cyclically reduced, and let $b_1, \dots, b_\ell\in G$ be elements such that the block decomposition of $g^h$ is $b_1^{t_1} \dots b_\ell^{t_\ell}$. Then
   $
   C_{G}(g) = \left(\Pi_{i=1}^{\ell} \langle b_i^{t_i}\rangle \times \langle link(b_1, \dots, b_\ell) \rangle\right)^{h^{-1}}.
   $
\end{theorem}

As already noted in Example \ref{ex:ap}, the following holds, and we will use it implicitly:
\begin{remark}
Let $v \in \Gamma$ be a vertex. Then the element $v$ is abelian-primitive in $G\Gamma$.
\end{remark}

\subsection{Main result for partially commutative groups (or RAAGs)}
We next turn to our main result about RAAGs, namely Theorem \ref{thm:DP_RAAGS}, which states that $\mc{DP}(G\Gamma, \ab)$ is undecidable in any nonabelian RAAG. 

\textbf{Outline of the main proof.}  
The proof of Theorem \ref{thm:DP_RAAGS} is an involved reprisal of the proof of undecidability of $\mc{DP}(F(\Sigma), \ab)$, and we emphasise the parallels between the two proofs in the next paragraphs. Our proof strategy is as follows. 

Recall that if $S\subseteq G$ is a subset of \ap \ elements of $G$, we let $R_S$ be the binary relation defined in $G$ by: $g, h\in G$ satisfy $R_S$ if and only if $|g|_s=|h|_s$ for all $s\in S$ (see Definition \ref{d: R1_R2}).  
The key lemma of this section is Lemma \ref{l: R_s2_PE_interp_in_RAAGS}. This states that if $S$ is a weak module, then $R_S$ is PE-definable in $G$. This is similar to the starting point for the proof of undecidability of $\mc{DP}(F(\Sigma), \ab)$, where we are able to PE-define $R_S$ in $(F(\Sigma), \ab)$ by means of centralisers (see the proof of Lemma \ref{lemma: DP_implications}). In the RAAG case, however, we are forced to deal with a whole weak minimal module $S$, rather than a single generator $s$. The proof is much simpler if we assume that  $G\Gamma$ has at least two distinct weak modules consisting each of a single vertex.

Next, we show that we can assume $G\Gamma$ has at least two distinct weak modules $S,T$ such that no vertex from $S$ is adjacent to a vertex in $T$. The weak modules $S$ and $T$ are similar to two generators $s,t$ in the free group case, where the subgroups $\langle s\rangle, \langle t\rangle$ served as proxies for $\mbb{Z}$. However, this cannot be identically replicated in a RAAG simply with $S$ and $T$ since the latter are sets of vertices. For this reason we take the ``diagonal generators of $S$ and $T$'', namely $h_1 = \prod_{v\in S} v$ and $h_2=\prod_{u\in T} u$, as we explain below. 

In free groups we are able to PE-define the sets $\langle s\rangle, \langle t\rangle$ via the centralisers of $s$ and $t$ and thus have ``access'' to the exponents of $s$ and $t$, which serve as our $\mbb{Z}$. In a RAAG the centraliser of $h_1$ (similarly for $h_2$) is more complicated: it has the form $\langle S \rangle \times \langle link(S) \rangle$, with $\langle S\rangle \cong \mbb{Z}^{|S|}$. Each one of these two factors presents a problem which has no parallel with the proof in free groups. The second factor is dealt with, roughly, by using the technical lemma Lemma \ref{l: technical_lemma}. The main idea is that the elements from  $\langle link(S) \rangle$ never use any vertex from $S$ or $T$, and thus eventually become irrelevant in our arguments. 
Regarding the first factor $\langle S\rangle$, to be able to read off the exponents of $h_i$ within it, $i=1,2$, we introduce the subgroup $G_{diag, S, T}$ defined as the set of elements $g$ of $G\Gamma$ that satisfy $R_{diag, S}$ and $R_{diag, T}$, where $g$ satisfies $R_{diag, S}$ if and only if $|g|_{s}=|g|_{s'}$ for all $s,s'\in S$ (i.e.\ if it is a power of one of the ``diagonal generators'' described above). In such a subgroup the centraliser of $h_i$ is $\langle h_i \rangle \times \langle link(S)\rangle$.

We finally show that $(G_{diag, S,T}, \ab)$ is PE-interpretable in $(G,\ab)$; thus it suffices to prove that $\mc{DP}(G_{diag, S,T}, \ab)$ is undecidable. At this point we have two \ap \ elements $h_1, h_2$ whose centralisers are almost as nicely behaved as in a free group, and the proof can be finished by applying the technical Lemma \ref{l: technical_lemma}, where its second condition is seen to be met in $G_{diag, S,T}$ by inspecting the centralisers $C_{G_{diag, S, T}}(h_1g)$, with $g$ ranging among  all $g\in G_{diag, S_1, S_2}$ satisfying $|g|_{h_1}=0$.


This concludes our intuitive explanation of the proof of Theorem \ref{thm:DP_RAAGS}. To prove this result formally will require the following several lemmas.

\begin{lemma}\label{l: R_s2_PE_interp_in_RAAGS}
	Let $G=G\Gamma$, and let $S$ be a weak module of $\Gamma$. Then the relation $R_{S}$ is PE-definable in $(G, \ab)$.
\end{lemma}
\begin{proof}
Let $R_S'$ refer to the unary relation satisfied by $g\in G$  if and only if $|g|_s = 0$ for all $s\in S$. 
    Since $|\cdot|_s$ is additive, $|k|_s=|h|_s$ for $k,h \in G$ if and only if $|kh^{-1}|_s=0$, so in order to show that $R_{S}$ is PE-definable in $(G, \ab)$ it suffices to prove that $R_S'$ is PE-definable in $(G,\ab)$. 

	First we claim that there exist vertices $u_1, \dots, u_n\in V\Gamma$ such that $$V\Gamma\setminus S = \bigcup_{i=1}^n star(u_i).$$
	 %
	Indeed, consider the set of vertices $V\Gamma\setminus star(S)$. For any $w \in V\Gamma\setminus star(S)$ we have that $star(w)\cap S=\emptyset$. Hence $$\bigcup_{u\in V\Gamma\setminus star(S)} star(u) \subseteq V\Gamma \setminus S.$$ Now we prove  that $V\Gamma \setminus S \subseteq \cup_{u\in V\Gamma\setminus star(S)} star(u)$: if $w\in V\Gamma$ does not belong to $star(S)$ then there is nothing to argue. If it does, then it belongs to $link(S) = star(S) \setminus S$. Next we show that  $star(w) \not \subseteq star(S)$. Once this is shown, we obtain that $w \in link(w')$ for any $w'\in star(w) \setminus star(S)$, and so $w\in\cup_{u\in V\Gamma\setminus star(S)} star(u)$, completing the proof of the first claim of the lemma: indeed, we can take $\{u_1, \dots, u_n\} = V\Gamma\setminus star(S)$. 
	
	To prove that $star(w) \not \subseteq star(S)$, note that  $star(w)$ cannot be properly contained in $star(S)$, since this would contradict the minimality of any of the vertices in $S$. Hence either  $star(w) \not\subseteq star(S)$ or $star(w)=star(S)$.  In the latter case, $w$ is a minimal vertex since $star(w)=star(S)=star(v)$ for any $v\in S$ and all such $v$'s are minimal. This would contradict the fact that $S$ is a weak module (since adding $w$ to $S$ would result in a larger set of minimal vertices all of them having the same star), and so we must have $star(w) \not\subseteq star(S)$, as required.
	
	 Our second claim is that  an element $x\in G$ satisfies $R_{S}'$ if and only if there exists $y \in \Pi_{i=1}^n C_G(u_i)$  such that  $\ab(x) = \ab(y)$, where the $u_i$'s are the vertices from the first claim of the present proof.  Indeed, due to the first claim and by Proposition \ref{p: centralizers_RAAGs} about centralisers in RAAGs, any element in $\Pi_{i=1}^n C_G(u_i) = \Pi_{i=1}^n \langle star(u_i)\rangle$ satisfies $R_{S}'$. Conversely, if an element $x$ belongs to $\langle V\Gamma \setminus S\rangle$, then  $x =  z_1 \dots z_{k_x}$ where $z_i \in V\Gamma \setminus S$ for all $i=1, \dots, k_x$, and again due to the first  claim of the lemma and Proposition \ref{p: centralizers_RAAGs}, we have that for each $z_i$ there exists $u_{j_i}$ such that $z_i \in C_G(u_{j_i})=\langle star(u_{j_i})\rangle$. Hence $x\in \Pi_{i=1}^{k_x} C_G(u_{j_i})$, and so $\ab(x) \in  \Pi_{i=1}^{n} \ab(C_G(u_{i}))$. It follows that $\ab(x) = \ab(y_1) \dots  \ab(y_n)$ for some $y_i\in C_G(u_{i})$ such that $\ab(y_i)\in \ab(C_G(u_{i}))$ ($i=1,\dots, n$). Thus it suffices to take $y=y_1\dots, y_n$.

	The lemma now follows immediately from the second claim and the fact that in any group the centraliser of an element is PE-definable.
\end{proof}

It is important that $G\Gamma$ does not decompose as a direct product. In this case the set of minimal vertices of $\Gamma$ satisfies certain favorable properties, as shown next.
\begin{lemma}\label{lem:two_modules}
	Let $\Gamma$ be a graph  and suppose that for any two weak modules $S, T$ we have $star(S)\cap T \neq \emptyset$, i.e.\ some vertex of $S$ is adjacent to some vertex of $T$ (this includes the case when $\Gamma$ has only one weak module). Then  $G\Gamma = \langle M \rangle \times \langle V\Gamma \setminus M \rangle$, where $M$ is the set of minimal vertices of $\Gamma$.
	
	In particular, if $G\Gamma$ is not a nontrivial direct product, then $\Gamma$ has at least two distinct  weak modules $S, T$ such that no vertex in $S$ is adjacent to a vertex in $T$.
\end{lemma}
\begin{proof}
    Suppose that for any two weak modules $S,T$, some vertex of $S$ commutes with some vertex of $T$. Since all vertices in a weak module have the same star, we have that in this case, for any two weak modules $S, T$, all vertices of $S$ commute with all vertices of $T$. Hence $M$, which consists of the union of all weak modules, forms a clique. 
    
	Now observe that  if $u$ is an arbitrary vertex of $\Gamma$, then there exists a minimal vertex $v\in M$ ($v=u$ if $u \in M$) such that $ star(v)  \subseteq star(u)$ since $\leq$ is a partial order in $V\Gamma$.  Since $M\subseteq star(v)$ we have  $M\subseteq star(u)$, so $u$ is adjacent to every vertex in $M$. The lemma follows. 
\end{proof}


\begin{lemma} \label{lem:C(gh)}
Let $S$ and $T$ be two distinct weak modules of $\Gamma$ such that no vertex in $S$ is adjacent to a vertex in $T$. Let $S_0\subseteq S$ and $T_0\subseteq T$ be two non-empty subsets of $S$ and $T$, and let $g=\prod_{v\in S_0} v$, $h=\prod_{u\in T_0} u$. Then $C_{G\Gamma}(gh) = \langle gh\rangle  \times \langle link(gh) \rangle$, with  $\langle link(gh) \rangle \subseteq \langle V\Gamma \setminus (S \cup T) \rangle$.
\end{lemma}
\begin{proof}
By the definition of $S_0$ and $T_0$ the element $gh$ is cyclically reduced. Moreover, there is no way to partition the support of $gh$, namely $S_0\cup T_0$, into two nonempty distinct subsets of vertices $A, B$, such that each vertex in $A$ commutes with $B$ and vice-versa. Hence by Theorem \ref{t: centralizers_RAAGS}, we have $C_{G\Gamma}(gh) = \langle gh\rangle \times \langle link(gh) \rangle$. The last part of the lemma follows from the assumption about $S$ and $T$ that no vertex in $S$ is adjacent to a vertex in $T$.
\end{proof}
\begin{definition}
If $X$ is a subset of \ap \ elements of $G$, let $R_{diag, X}$ be the unary relation defined as: $g\in G$ satisfies $R_{diag, X}$ if and only if $|g|_{v}=|g|_{u}$ for any two $v,u\in X$.
\end{definition}
Note that if $X$ has only one vertex then all $g \in G$ satisfy the relation $R_{diag, X}$.

\begin{lemma}\label{l: diag_PE_def}
    Let $\Gamma$ be a graph  such that $G\Gamma$ does not decompose as a nontrivial direct product. Let $S$ and $T$ be two distinct weak modules of $\Gamma$ such that no vertex from $S$ is adjacent to $T$ (these exist by Lemma \ref{lem:two_modules}). Then the relations $R_{diag, S}$ and $R_{diag, T}$ are PE-definable in $(G\Gamma,\ab)$.
\end{lemma}
\begin{proof}
 Let $w = \prod_{v\in S} v$ and let $u \in T$.  Then $C(wu)= \langle wu\rangle \times \langle link(wu)\rangle= \{(wu)^t \mid t \in \mathbb{Z}\}\times \langle link(wu)\rangle$ by Lemma \ref{lem:C(gh)}. Notice that all elements in $\langle wu\rangle $ satisfy $R_{diag, S}$ and that no vertex from $S\cup T$ appears in the support of any element from $\langle link(wu)\rangle$. It follows that an element $g\in G\Gamma$ satisfies $R_{diag, S}$ if and only if there exists $h\in C(wu)$ such that $|g|_s = |h|_s$ for all $s\in S$, i.e.\ if $g,h$ satisfy the relation $R_S$. Since $R_S$ is PE-definable in $(G\Gamma, \ab)$ by Lemma \ref{l: R_s2_PE_interp_in_RAAGS}, $R_{diag, S}$ is also PE-definable in $(G\Gamma, \ab)$.
\end{proof}

Given two distinct weak modules $S_1, S_2$ of $\Gamma$, we let $G_{diag, S_1, S_2}$ be the subset of $G=G\Gamma$ formed by those elements of $G$ satisfying the unary relations $R_{diag, S_1}, R_{diag, S_2}$:
$$G_{diag, S_1, S_2}=\{g \in G \mid R_{diag, S_i}(g) \ \textrm{holds for} \, i=1,2\}.$$
Recall $R_S, R_{s_1, s_2}$ and $R_{diag, S}$ from Definition \ref{d: R1_R2}.

\begin{lemma}\label{lem:ap_elements}
    Let $\Gamma$ be a graph such that $G=G\Gamma$ does not decompose as a nontrivial direct product, and let $S_1, S_2$ be two distinct weak modules of $\Gamma$ such that no vertex from $S_1$ is adjacent to a vertex of $S_2$. Then 
    \begin{enumerate}
    \item $G_{diag, S_1, S_2}$ is a PE-definable subgroup of $(G, \ab)$, 
    \item the elements $h_1=\prod_{w\in S_1} w$ and $h_2=\prod_{w\in S_2} w$ are \ap \ in $G_{diag, S_1, S_2}$, and
    \item the relations $R_{h_1}, R_{h_2}, R_{h_1, h_2}$  are PE-definable in $(G_{diag, S_1, S_2}, \ab)$.

   \end{enumerate} 
\end{lemma}

\begin{proof}
\noindent
\begin{enumerate}

\item The set $G_{diag, S_1, S_2}$ is closed under multiplication, and thus a subgroup, by the additivity of the maps $|\cdot|_s$ (Lemma \ref{l: morphism}); furthermore, $G_{diag, S_1, S_2}$ is PE-definable in $(G,\ab)$ by Lemma \ref{l: diag_PE_def}. 

\item Next we show that $h_i=\prod_{w\in S_i} w$ is \ap \ in $G_{diag, S_1, S_2}$ for $i=1,2$. Let $\ab$ be the natural projection of $G$ onto its abelianisation. We first note that any element $g$ of $G_{diag, S_1, S_2}$ can be written uniquely in the following form:
$$
g=h_1^{n_1(g)} h_2^{n_2(g)} \left(\prod_{v\in V\Gamma \setminus (S_1\cup S_2)} v^{n_v(g)} \right) c
$$
for some $n_1(g), n_2(g), n_v(g)\in \mbb{Z}$ ($v\in V\Gamma\setminus (S_1\cup S_2)$) and some $c\in G'$ (this can be seen by projecting $g$ onto $\ab(G)$ and rearranging the generators).

Let $\ab_{diag, S_1, S_2}$ be the projection of $G_{diag, S_1, S_2}$ onto its abelianisation. We claim that $$
\langle \ab_{diag, S_1, S_2}(h_1)\rangle \cap \langle \ab_{diag, S_1, S_2}\left(\{h_2\}\cup V\Gamma\setminus(S_1\cup S_2)\right) \rangle = 1.
$$ Note that once this claim is proved we will have shown that $h_1$ is \ap. 
To prove the claim, note that for any $\ab_{diag, S_1, S_2}(g)$ in this intersection we have 
$$
h_1^{-n_1(g)} = h_2^{n_2(g)} \left( \Pi_{V\Gamma\setminus(S_1\cup S_2)} v^{n_v(g)} \right) c
$$
for some $c\in ker(\ab_{diag, S_1,  S_2}) = G_{diag, S_1, S_2}' \subseteq G'$.
The additivity of the maps $|\cdot|_u$ ($u\in V\Gamma$) now implies that we must have $n_1(g)=n_2(g)=n_v(g)=0$ for all $v$. The claim follows. The argument for $h_2$ is analogous.
%


\item That $R_{h_i}$ ($i=1,2$) is PE-definable in $(G_{diag, S_1, S_2}, \ab)$ follows from the fact that a pair of  elements $g,h\in G_{diag, S_1, S_2}$   satisfies $R_{h_i}$ if and only if they satisfy $R'_{h_i}$, as defined at the beginning of the proof Lemma \ref{l: R_s2_PE_interp_in_RAAGS}; by that proof $R'_{h_i}$ is PE-definable in $(G, \ab)$, and since $G_{diag, S_1, S_2}$ is PE-definable in $(G, \ab)$ we have that $R'_{h_i}$, and therefore $R_{h_i}$, is also PE-definable in $(G_{diag, S_1, S_2}, \ab)$.



To PE-define the relation $R_{h_1, h_2}$ in  $(G_{diag, S_1, S_2}, \ab)$, observe that two elements $x,y\in G_{diag, S_1, S_2}$ satisfy this relation if and only if  there exists an element $z\in \langle h_1h_2\rangle \times \langle link(h_1h_2)\rangle$ such that the pair $x,z$ satisfies $R_{h_1}$ and the pair $y,z$ satisfies $R_{h_2}$ (this is because $|z|_{h_1} = |z|_{h_2}$ for all $z\in \langle h_1h_2\rangle\times \langle link(h_1h_2)\rangle$, since no vertex from $S_1\cup S_2$ belongs to $link(h_1h_2)$ by assumption).  

Next we claim that  $$\langle h_1h_2\rangle \times \langle link(h_1h_2)\rangle= C_{G_{diag, S_1, S_2}}(h_1h_2).$$ Indeed, $h_1h_2$ is cyclically reduced, and because any vertex in $S_1$ is not adjacent to any vertex in $S_2$, the block decomposition of $h_1h_2$ consists of one block. Then the claim follows from Theorem \ref{t: centralizers_RAAGS}.
 The claim shows that $\langle h_1h_2\rangle\times \langle link(h_1h_2)\rangle$ is PE-definable in $(G_{diag, S_1, S_2}, \ab)$. This, together with the fact that $R_{h_1}$ and $R_{h_2}$ are PE-definable in $(G_{diag, S_1, S_2}, \ab)$, implies that $R_{h_1, h_2}$ is also PE-definable in $(G_{diag, S_1, S_2}, \ab)$, as required.
\end{enumerate}
\end{proof}

%
%

Combining all the lemmas in this section leads to the main result for RAAGs (see the outline of the proof at the beginning of this section):
 
\begin{theorem} \label{thm:DP_RAAGS}
	Let $G=G\Gamma$ be a RAAG such that $G$ is not   abelian. Then $\mc{DP}(G, \ab)$ is undecidable.
\end{theorem}
\begin{proof}

	We proceed by induction on the number $n$ of vertices of $\Gamma$. The base case $n=2$ only allows for $G$ to be a free group, and by Theorem \ref{thm:free_gp} $\mc{DP}(G, \ab)$ is undecidable in this case. 
	Assume $n>2$ and $G$ does not decompose as a nontrivial direct product; then $\Gamma$ has two distinct  weak modules  $S_1, S_2$ such that no vertex from $S_1$ is adjacent to a vertex from $S_2$ by Lemma \ref{lem:two_modules}. 
	
	We will apply Lemma \ref{l: technical_lemma} to the subgroup $G_{diag, S_1, S_2}$ and the \ap \ elements $h_1, h_2$, where $h_1=\prod_{w\in S_1}w$,  $h_2=\prod_{w\in S_2}w$. By Proposition \ref{prop:DPred} (on reduction of Diophantine problems) the undecidability of $\mc{DP}(G_{diag, S_1, S_2}, \ab)$ will imply the undecidability of $\mc{DP}(G, \ab)$ since $G_{diag, S_1, S_2}$ is PE-definable in $(G, \ab)$ (Lemma \ref{lem:ap_elements}). 

Condition (1) of Lemma \ref{l: technical_lemma} is satisfied because the relations $R_{h_1}, R_{h_2}$, and $R_{h_1, h_2}$ are PE-definable in $(G_{diag, S_1, S_2}, \ab)$ by Lemma \ref{lem:ap_elements}(3).  
		We next establish that condition (2) of Lemma \ref{l: technical_lemma} is satisfied. To simplify the notation we use $G_d$ instead of $G_{diag, S_1, S_2}$ throughout the rest of the proof.
	Let $K_{h_i} \leq G_d$ be the subgroup consisting of those $g\in G_{d}$ such that $|g|_{h_i}=0$ ($i=1,2$), and let $K_{h_1, h_2} = K_{h_1}\cap K_{h_2}$.  Observe that these subgroups are all PE-definable in $(G_d, \ab)$ because the relations $R_{h_i}$ are. 
	In order to apply Lemma \ref{l: technical_lemma} to $G_d$ it remains to see that there exists a disjunction of systems of equations $\bigvee_{i=1}^{m}\Sigma_i(z, x, y_1, \dots, y_n)$ on variables $z,x,y_1, \dots, y_n$, such that  for any $b\in K_{h_1}$ there exists $h_b\in G_d$ so that  $\bigvee_{i=1}^{m}\Sigma_i(b, x, y_1, \dots, y_n)$ 
	PE-defines the set $h_b^{-1}\{( h_1h_2^{|b|_{h_2}})^t\mid t\in \mbb{Z}\}K_{h_1,h_2}h_b$ in $(G_d, \ab)$.	To prove this, we will  show that 
	\begin{equation}\label{set_eq}
	C_{G_d}(h_1b)K_{h_1,h_2} = h_b^{-1}\{( h_1h_2^{|b|_{h_2}})^t\mid t\in \mbb{Z}\}K_{h_1,h_2}h_b,
	\end{equation}
	for a certain $h_b$. That is, modulo $K_{h_1, h_2}$, the centraliser of $h_1b$ is a conjugate of the set $\{( h_1h_2^{|b|_{h_2}})^t\mid t\in \mbb{Z}\}$. Once this is proved, we can take the expression $$x = x_1x_2 \wedge [x_1, h_1b]=1 \wedge x_2\in K_{h_1,h_2}$$ and rewrite it as a disjunction of system of equations (for this, we replace $x_2\in K_{h_1,h_2}$ for the disjunction of systems of equations that PE-define $K_{h_1,h_2}$ in $(G_d, \ab)$, obtaining a positive existential formula; then  we convert this formula into an equivalent disjunction of systems of equations).

	Let $b\in K_{h_1}$, and let $h_b\in G_d$ be such that $h_b^{-1}(h_1b)h_b= (h_1b)^{h_b}$ is cyclically reduced. Let $b_1^{n_1} \dots b_r^{n_r}$  be the block decomposition of $(h_1b)^{h_b}$. By Theorem \ref{t: centralizers_RAAGS}  
	\begin{align*}
	C_{G_d}(h_1b) &= C_G(h_1b) \cap  G_d = (C_{G}(b_1^{n_1} \dots b_r^{n_r}))^{h_b^{-1}} \cap G_d\\ &=\left(\prod_{i=1}^{r }\langle b_i^{n_i}\rangle_G \times \langle link( b_1 \dots b_r) \rangle_G\right)^{h_b^{-1}} \cap G_d,
	\end{align*}
	where by $\langle \cdot \rangle_G$ we mean generation in the group $G$ (as opposed to $G_d$).

	Since $b\in K_{h_1} \leq G_d$, no vertex from $S_1$ belongs to the support of $b$, so $|(h_1b)^{h_b}|_{h_1}=1$ by the additivity of $|\cdot|_{h_1}$. Moreover, by the definition of weak module, a vertex $u\in V\Gamma$  is adjacent to a vertex $v\in S_1$ if and only if $u$ is adjacent to all vertices of $S_1$. This means that one of the two following scenarios hold: 
	\begin{enumerate}
	\item[(i)] If $supp((h_1b)^{h_b})$ contains a vertex not in $S_1$ and not adjacent to any vertex $u$ in $S_1$, then all vertices from $S_1$ appear in the same block, say $b_1$, together with the vertex $u$. In particular, if some vertices of $S_2$ belong to $supp((h_1b)^{h_b})$, then they appear only in the block $b_1$ (because by assumption no vertex from $S_2$ is adjacent to a vertex from $S_1$). 
	\item[(ii)] If, on the contrary, all vertices not in $S_1$ in the support of $(h_1b)^{h_b}$ commute with some  vertex of $S_1$, then the support of $(h_1b)^{h_b}$ forms a clique in $\Gamma$ and $S_2 \cap supp((h_1b)^{h_b})=\emptyset$. 
	\end{enumerate}
	
	In the first case (Item (i)) we can write $b_1 = h_1 h_2^{|b_1|_{h_2}}k$ for some $k\in K_{h_1, h_2}$ 
	(this is because  we can project $b_1$ onto the abelianisation $G/G'$, order the  $h_1, h_2$'s towards the left, and ``pull'' the resulting expression back into $G_d$ by adding a commutator element, obtaining $b_1 = h_1 h_2^{|b_1|_{h_2}} r c$, where $r\in K_{h_1, h_2}$ and $c\in G'$; now since $G' \leq K_{h_1, h_2}$ we obtain the desired expression). 
	This way $(h_1b)^{h_b} = (h_1h_2^{|b_1|_{h_2}}k)^{n_1} b_2^{n_2} \cdots b_r^{n_r}$. Note that $n_1 =1$ because $|h_1b|_{h_1}=1$. Since all vertices from $S_2$ belonging to $supp((h_1b)^{h_b})$ must appear in the first block, we have  $b_2^{n_2} \cdots b_r^{n_r}\in K_{h_1, h_2}$(\textdagger). From this  it follows that $C_{G}(h_1 b) \leq G_d$, and so  $C_{G_d}(h_1 b) = C_{G}(h_1 b)$. Temporarily writing $L=\langle link( b_1 \dots b_r)\rangle_G$ we obtain
	\begin{align*}
	C_{G_d}(h_1b) K_{h_1,h_2}  &= C_{G}(h_1b) K_{h_1,h_2} \\
	&\overset{*}= \left(\left(\prod_{i=1}^{r }\langle b_i^{n_i}\rangle_G \times L\right) K_{h_1,h_2}\right)^{h_b^{-1}} \\ &= \left(\prod_{i=1}^{r }\langle b_i^{n_i}\rangle_G L K_{h_1, h_2} \right)^{h_b^{-1}} \\&\overset{**}= \left(\langle b_1\rangle K_{h_1,h_2} \right)^{h_b^{-1}} \\ &=  \left(\langle h_1h_2^{|b_1|_{h_2}}k\rangle K_{h_1,h_2}\right)^{h_b^{-1}} = \left(\langle h_1 h_2^{|b_1|_{h_2}} \rangle K_{h_1,h_2}\right)^{h_b^{-1}},
	\end{align*}
	where for $\overset{*}=$ we used the normality of $K_{h_1,h_2}$ to move the conjugator $h_b^{-1}$ outside the expression, and for $\overset{**}=$ we used (\textdagger), plus the fact that no vertex from $S_1$ is adjacent to any vertex from $S_2$, which implies $LK_{h_1, h_2} = K_{h_1, h_2}$.
	This completes the Item (i). 
	
	In the second case (Item (ii)) each vertex in $supp((h_1b_1)^{h_b})$ is its own block. In this case, letting $U:=supp((h_1b_1)^{h_b})\setminus (S_1\cup S_2) = supp((h_1b)^{h_b}) \setminus S_1$ gives 
	%
	%
		\begin{align*}C_{G_d}(h_1b)K_{h_1, h_2}&= \left( \left(\prod_{v\in S_1} \langle v \rangle_G \prod_{u\in U} \langle u \rangle_G \times L \right)^{h_b^{-1}} \bigcap G_d\right) K_{h_1, h_2} \\&= \left(\prod_{v\in S_1} \langle v \rangle_G \prod_{u\in U} \langle u\rangle_G K_{h_1, h_2}\right)^{h_b^{-1}} \bigcap G_dK_{h_1, h_2} \\ &=\left(\prod_{v\in S_1} \langle v \rangle_GK_{h_1, h_2}\right)^{h_b^{-1}} \bigcap G_dK_{h_1, h_2} \\ &= h_b\langle h_1 \rangle_{G}K_{h_1, h_2}h_b^{-1},\end{align*}
	where the last equality follows from the aditivity of the maps $|\cdot|_{s_i}$, and where we have used again that $LK_{h_1, h_2} = K_{h_1, h_2}$. This completes the proof of Item (ii); moreover, we have proved that the second part of Lemma \ref{l: technical_lemma} holds for $G_d$.

	Finally, observe that $|b_1|_{h_2} = |b|_{h_2}$, which proves (\ref{set_eq}).
Hence Lemma  \ref{l: technical_lemma} applies and therefore $\mc{DP}(G_{diag, S_1, S_2}, \ab)$ is undecidable. Thus $\mc{DP}(G, \ab)$ is undecidable if $\Gamma$ has more than two vertices and $G$ does not decompose as a nontrivial direct product.
	
	In the case where $G$ decomposes as a nontrivial direct product, $G=G\Gamma \cong G\Delta_1 \times G\Delta_2$ for some graphs $\Delta_1, \Delta_2$, each with strictly fewer vertices than $\Gamma$. Since $G$ is assumed to not be free abelian, either $G\Delta_1$ or $G\Delta_2$ is not free abelian; now we can use the induction hypothesis to obtain that $\mc{DP}(G\Delta_1, \ab)$  or $\mc{DP}(G\Delta_2, \ab)$ is undecidable. Then $\mc{DP}(G, \ab)$ is undecidable by Lemma \ref{l: DPab_in_direct_products}.
\end{proof}

\section{Hyperbolic groups with first Betti number $\geq 2$}\label{sec:hyp}

In this section we study the problem $\mc{D}(G, \ab)$ when $G$ is a hyperbolic group. We show that if the abelianisation of $G$ has (torsion-free) rank $\geq 2$ (or equivalently, $b_1(G)\geq 2$) then this problem is undecidable. The case when the rank is $0$ (equivalently, $G$ has finite abelianisation) is covered in Section \ref{sec:fin_ab}. In the latter case  $\mc{D}(G, \ab)$  is decidable. 

The case when  $G$ has $b_1(G)=1$ remains an open problem.

\subsection{Preliminary results on PE-interpretability}

The following lemmas establish some basic results that will be needed  in the present section.
\begin{lemma}\label{l: f_i_interpret}
Let $H\leq K \leq G$ be subgroups of a group $G$. Suppose that $H$ has finite index in $K$ and $H$ is PE-interpretable in $G$. Then $K$ is PE-interpretable in $G$.
\end{lemma}
\begin{proof}
Let $k_1, \dots, k_t \in K$ be such that  $K$ is the disjoint union of the cosets  $k_1H, \dots, k_{t}H$. Each of these cosets is PE-interpretable in $G$, so $K$ is also PE-interpretable in $G$.
\end{proof}


Recall that a verbal subgroup of a group $G$ is a subgroup $H $ for which there exists a word $w(x_1, \dots, x_n) \in F(x_1, \dots, x_n)$ such that $H=\langle w(g_1, \dots, g_n) \mid g_1, \dots, g_n\in G \rangle$. Such group is said to have finite verbal width if there exists an integer $N$ such that all elements of $H$ can be written as a product of at most $N$ elements of the form $w(g_1, \dots, g_n)$ or  $w(g_1, \dots, g_n)^{-1}$, ($g_1, \dots, g_n\in G$). It is well-known \cite[Section 2.3]{GMO} that any subgroup $H$ of $G$ with finite verbal width is definable in $G$ by a single equation. In particular, it is PE-definable and PE-interpretable in $G$. We will  use this fact in the next lemma.

\begin{lemma}\label{l: fi_PE}
	Let $G$ be a virtually polycyclic group, and let $H$ be a finite index subgroup of $G$. Then $H$ is PE-interpretable in $G$.
\end{lemma}
\begin{proof}

Let $s$ be the index of $H$ in $G$, and let $G^s$ be the verbal subgroup of $G$ given by the word $w(x)=x^s$, i.e.\ the subgroup generated by $\{g^s\mid g\in G\}$. 

Since $G$ is verbally elliptic (i.e.\ all verbal subgroups of $G$ have finite verbal width) \cite{Segal_book}, we have that $G^s$ is e-interpretable in $G$ \cite[Section 2.3]{GMO}. Note also that $G^s$ is normal in $G$ and that $A=G/G^s$ is a virtually polycyclic group that is periodic, i.e. with all elements of finite order
We claim that $A=G/G^s$ is finite, as follows. By Lemma \ref{l: f_i_interpret}, it is enough to assume that $A$ is polycyclic. Let $d$ be the length of a subnormal series $A=A_d > A_{d-1} > \dots >A_0=1$ of $A$ with cyclic factors. The statement is clear if $d=1$. If $d> 1$ then $A_{d-1}$ is finite by induction, and moreover $A/A_{d-1}$ is finite because $A$ is periodic, hence $A/A_{d-1}$ is also periodic, and because $A/A_{d-1}$ is cyclic. Hence $A$ is a finite-by-finite group, hence finite.

Thus $G^s$ has finite index in $G$, and therefore in $H$. Now the subgroups $G^s \leq H\leq G$ satisfy the hypotheses of Lemma \ref{l: f_i_interpret}. Hence $H$ is PE-interpretable in $G$. 
%
\end{proof}

\begin{lemma}\label{l: PE-interpret_powers}
Let $G$ be a group and let $g\in G$ be an element of infinite order. Suppose the centralizer of $g$ is virtually $\mathbb{Z}$. Then $\langle g\rangle \cong \mbb{Z}$ is PE-interpretable in $G$.
\end{lemma}
\begin{proof}
	Since $\langle g\rangle \cong \mbb{Z}$ and $C_G(g)$ is virtually infinite cyclic, $\langle g\rangle$ is a finite index subgroup of $C_G(g)$, and so $\langle g\rangle$ is  definable in $C_G(g)$ by positive existential formulas by Lemma \ref{l: fi_PE}.  Since  $C_G(g)$ is e-definable in $G$, we have that $\langle g\rangle$ is PE-definable in $G$.
\end{proof}

Recall that by $R_s$ we denote the binary relation satisfied by two elements $g,h$ such that $|g|_s=|h|_s$, for $s$ an abelian-primitive element.
\begin{lemma}\label{l: R1_R2}
	Let $G$ be a group with finite generating set $S$. Suppose all generators have infinite order, and virtually $\mathbb{Z}$ centralisers.
	Then for any \ap \ $s \in S$ the relation $R_{s}$ is PE-interpretable in $(G, \ab)$.


\end{lemma}
\begin{proof}
Let $s \in S$ be \ap. By Lemma \ref{l: morphism}, two elements $g,h\in G$ satisfy $R_{s}(g,h)$ if and only if $|gh^{-1}|_s=0$. 
Hence it suffices to prove that the predicate $|x|_s=0$ for $x\in G$
is PE-interpretable in $(G,\ab)$. 

Note that an element $x\in G$  satisfies $|x|_s=0$ if and only if there exists $y\in \prod_{t\in S\setminus \{s\}} \langle t \rangle$ such that $\ab(x)=\ab(y)$. Since for every $t \in S$ the group $\langle t \rangle$ is PE-interpretable in $G$ by Lemma \ref{l: PE-interpret_powers} and $S$ is finite, we get that $\prod_{t\in S\setminus \{s\}} \langle t \rangle$ is PE-interpretable in $G$.  Hence, 
 $\{x\in G \mid |x|_s=0\}$ is PE-interpretable in $(G,\ab)$, and so is $R_s$. \end{proof}

\begin{lemma}\label{l: PE-define_relation}
Let $G$ and $S$ be as in Lemma \ref{l: R1_R2}. Let $s, t \in S$ be two \ap \ elements with $\langle \ab(s), \ab(t) \rangle \cong \mbb{Z}^2$, and assume the centraliser of $st$ is virtually $\mathbb{Z}$. 
	
	Then $R_{s,t}(\cdot,\cdot)$ is PE-interpretable in $(G, ab)$.
\end{lemma}
\begin{proof}
By Remark \ref{l:product} $st$ has infinite order. It follows from Lemma \ref{l: morphism} that $g,h\in G$ satisfy $R_{s,t}(\cdot,\cdot)$ if and only if there exists some $x \in \langle s t \rangle$ such that $ |g|_s = |x|_s$ and $|h|_t = |x|_t$.
	
	The result follows since the relations $|g|_s = |x|_s$ and $|h|_t = |x|_t$ are PE-interpretable in $(G,\ab)$ by Lemma \ref{l: R1_R2} and because  $\langle s t\rangle$ is PE-definable in $G$ by Lemma \ref{l: PE-interpret_powers}. 
\end{proof}

\subsection{Main results on hyperbolic groups}
We begin with some preliminary results.
In the proof of Lemma \ref{l: condition_2_for_hyperbolic} we will use two related, but not equivalent notions, those of \emph{primitive} and \emph{\ap} elements (the latter introduced in Definition \ref{def:ap}). A \emph{primitive} element in a hyperbolic group is one of infinite order that is not a non-trivial power of another element. This terminology is relatively standard in the literature, although sometimes `strongly primitive' is used instead of `primitive.'

\begin{remark}\label{rmk:ap&p}
An \ap \ element (see Definition \ref{def:ap}) is in fact always primitive: it clearly has infinite order, and it cannot be a non-trivial power because its abelianisation would not generate a direct factor of $G^{\ab}$. However, a primitive element is not always \ap, since its abelianisation might be of finite order.
\end{remark}

The following result is well-known (see, for example, \cite[Theorem 4.3]{Osin}). 
\begin{lemma}\label{l: centralizers_in_hyp_groups}
 Let $G$ be a hyperbolic group, and let $g\in G$ be an element of infinite order. Then the centraliser $C(g)$ is virtually $\mathbb{Z}$. If, additionally, $g$ is a primitive element, then the centraliser $C(g) = \langle g \rangle \times T_g$ for some finite group $T_g$.
\end{lemma}

The following lemma is essential for establishing Theorem \ref{thm:hyp}.

\begin{lemma}\label{lem:basis_hyp}
Let $G$ be a hyperbolic group. Then there exists a basis $S$ of $G$ consisting entirely of loxodromic (infinite-order) elements. 
\end{lemma}
\begin{proof}
Fix any finite generating set $X=\{x_1, \dots x_n\}$ of $G$. By \cite{AM2007}, hyperbolic groups have the property $P_{naive}$: for any finite set $F\subset G\setminus \{1\}$ and element $f\in F$ there exists an element $g_F\in G$ of infinite order such that for each $f\in F$ the subgroup $\langle f, g_F \rangle$, generated by $f$ and $g_F$, is canonically isomorphic to the free product $\langle f\rangle \ast \langle g_F\rangle$.

Let $g_0:=g_X$ be the infinite order element we can associate to the finite generating set $X$ to satisfy $P_{naive}$, and let $g_i=g_0x_i$, $i=1, \dots, n$ (we may instead let $g_i=x_i$ if $x_i$ has infinite order, and $g_i=g_0x_i$ otherwise). Since each $\langle x_i, g_0 \rangle \cong \langle x_i\rangle \ast \langle g_0\rangle$ is a free product, the element $g_i$ must have infinite order, as it does not belong to any of the factor groups. It is now easy to see that the set $S=\{g_0, g_1, \dots, g_n\}$ is a generating set consisting only of infinite order elements: we have already showed all $g_i$ have infinite order, and since we can write every initial generator $x_i \in X$ as 
$x_i=(g_0)^{-1}g_i$, we can express any element in $G$ in terms of the elements in $S$.
%
\end{proof}

Next, Lemma \ref{l: condition_2_for_hyperbolic} establishes condition (2) of Lemma \ref{l: technical_lemma} for hyperbolic groups. Recall that for any $\{s_1,\dots, s_n\} \subset G$, the set $K_{s_1, \dots, s_n}:=\{g \in G \mid |g|_{s_i}=0, \forall i=1,\dots, n\}$ is a normal subgroup of $G$ (see Lemma \ref{lem:Ks_def}). Furthermore, in the proof of Lemma \ref{l: condition_2_for_hyperbolic} we consider the supremum of the  finite subgroup sizes in $G$, given by \begin{equation}\label{eq:finite}
M=M(G):=\sup \{|F| \,\mid\, |F| < \infty, F \leq G \}.
\end{equation}
By \cite[III.$\Gamma$.Theorem 3.2]{BridsonHaefliger}, $M$ is finite.

\begin{lemma}\label{l: condition_2_for_hyperbolic}
 Let $G$ be a hyperbolic group with  generating set $S$ and suppose all generators in $S$ have infinite order. 
 Let $s_1, s_2 \in S$ be two \ap \ elements with $\langle \ab(s_1), \ab(s_2) \rangle \cong \mbb{Z}^2$.

  Then there exists a disjunction of systems of equations $$\bigvee_{i=1}^{m}\Sigma_i(z, x, y_1, \dots, y_n)$$ on variables $z,x,y_1, \dots, y_n$, such that  for all $g\in K_{s_1}$  the disjunction of systems of equations $ \bigvee_{i=1}^{m}\Sigma_i(g, x, y_1, \dots, y_n)$ PE-defines the set $$\{(s_1  s_2^{|g|_{s_2}})^t\mid t\in \mbb{Z}\} K_{s_1,s_2}$$ in $(G, \ab)$.
 %
\end{lemma}

\begin{proof}
Fix $g\in K_{s_1}$, that is, $|g|_{s_1}=0$. We claim that $s_1g$ is a primitive element. Indeed, the projection of $\ab(s_1g)$ onto $\langle\ab(s_1)\rangle$ is $\ab(s_1)$, which is of infinite order because $s_1$ is \ap. So $\ab(s_1g)$, and thus $s_1g$, has infinite order. Moreover, if $s_1g$ were a non-trivial power of another element, then the projection of $\ab(s_1g)$ onto $\langle \ab(s_1)\rangle$ would be a non-trivial power. But this projection is $\ab(s_1)$, as discussed above; since $s_1$ is \ap, $s_1$ must be primitive by Remark \ref{rmk:ap&p}, so $\ab(s_1)$ cannot be a power, and neither can $s_1g$. This proves the claim.

Since $s_1g$ is a primitive element, Lemma \ref{l: centralizers_in_hyp_groups} implies that $C(s_1g) = \langle s_1g \rangle \times T_{s_1g}$ for some finite subgroup $T_{s_1g}$. Note that $C(s_1 g)^{|T_{s_1g}|} = \langle (s_1g)^{|T_{s_1g}|}\rangle$ and
$$
\langle s_1g \rangle = \bigcup_{i= 0}^{|T_{s_1g}|-1} (s_1g)^i \langle (s_1g)^{|T_g|}\rangle = \bigcup_{i= 0}^{|T_{s_1g}|-1} (s_1g)^i C(s_1g)^{|T_{s_1g}|}.
$$
Hence $\langle s_1g \rangle$ is PE-defined in $G$ by the following disjunction of systems of equations on variables $x, z$:
$$
\bigvee_{i=0}^{|T_{s_1g}|-1} \left((x=(s_1g)^iz^{|T_{s_1g}|}) \wedge ([z,s_1g]=1)\right).
$$
Let $M$ be as in (\ref{eq:finite}). Then $M\geq |T_{s_1g}|$, and the following disjunction of equations still PE-defines $\langle s_1g \rangle$ in $G$:
$$
\bigvee_{i=0}^{M} \left((x=(s_1g)^iz^{|T_{s_1g}|}) \wedge ([z,s_1g]=1)\right).
$$

 Since $G'\leq K_{s}$ for any \ap\ element $s$, for each $g\in K_{s_1}$ we can write a unique expression of the form $g = s_2^{|g|_{s_2}} r$ for some $r\in K_{s_1,s_2}$ (indeed, we can project $g$ onto $G/G'$, arrange all generators together, and then ``pull back''  $g$ into $G$ by adding a commutator, which belongs to  $K_{s_1,s_2}$, to the right-hand side).
 Now we have
 $$
  \{(s_1 g)^t \mid t\in \mbb{Z}\} = \{(s_1 s_2^{|g|_{s_2}}r)^t \mid t\in \mbb{Z}\}.
 $$
 And so, since $K_{s_1, s_2}$ is normal,
  $$
  \{(s_1 g)^t \mid t\in \mbb{Z}\}K_{s_1,s_2} = \{(s_1 s_2^{|g|_{s_2}})^t \mid t\in \mbb{Z}\}K_{s_1,s_2}.
 $$
%
By Lemma \ref{l: R1_R2}, since both $s_1, s_2$ are  \ap\ elements and all generators of $G$ are assumed to have infinite order (and all centralisers of primitive elements in a hyperbolic group are virtually $\mbb{Z}$), it follows that both relations $R_{s_1}, R_{s_2}$ are PE-interpretable in $(G, \ab)$. Furthermore, Lemma \ref{l: PE-define_relation} implies that the relation $R_{s_1, s_2}$ is PE-interpretable in $(G,\ab)$ since $s_1s_2$ has infinite order (this is a particular case of the first claim in the present proof) and virtually $\mbb{Z}$ centraliser. This shows that $K_{s_1,s_2}$ is PE-definable in $(G, \ab)$.

 Let $\Phi(w, w_1, \dots, w_m)$ be a disjunction of systems of equations PE-defining  $K_{s_1,s_2}$ in $G$, so that $g\in G$ belongs to $K_{s_1, s_2}$ if and only if $\Phi(g, w_1, \dots, w_m)$ has a solution on the variables $w_1,\dots, w_m$. Then the following formula, once written in an equivalent form of a disjunction of systems of equations, PE-defines $\{(s_1 s_2^{|g|_{s_2}})^t \mid t\in \mbb{Z}\}K_{s_1,s_2}$ in $(G, \ab)$: 
$$
\left[\bigvee_{i=0}^{N} \left((x=(s_1g)^iz^{|T_{s_1 g}|} w) \wedge ([z,s_1g]=1)\right)\right] \wedge \Phi(w, w_1, \dots, w_m).
$$
%
\end{proof}

\begin{theorem}\label{thm:hyp}
Let $G$ be a hyperbolic group with infinite abelianisation of torsion-free rank $\geq 2$. Then $\mc{DP}(G, ab)$ is undecidable.
\end{theorem}

\begin{proof}
By Lemma  \ref{lem:basis_hyp} we can find a basis $S$ of $G$ with all generators of infinite order. Moreover, suppose that $s_1, s_2$ are two elements in $G$ such that their abelianisation generates a free rank-$2$ group in the abelianisation $G^{\ab}$, that is, $\langle \ab(s_1), \ab(s_2) \rangle \cong \mbb{Z}^2$. Such elements exist because $b_1(G)\geq 2$ by hypothesis. 
Add these two elements to the generating set $S$ and consider the new generating set $X=S \cup \{s_1,s_2\}$. 


Finally we apply Lemma \ref{l: technical_lemma} to $G$. By Lemma \ref{l: centralizers_in_hyp_groups} all the centralisers of the generators in the set $X$ are virtually $\mathbb{Z}$, so we can apply Lemmas \ref{l: R1_R2} and \ref{l: PE-define_relation} to the abelian - primitive elements $s_1, s_2$; these imply that $R_{s_1}, R_{s_2}, R_{s_1, s_2}$ are PE-interpretable in $(G,\ab)$, establishing part (1) of Lemma \ref{l: technical_lemma}. Then Lemma \ref{l: condition_2_for_hyperbolic} guarantees that condition (2) of Lemma \ref{l: technical_lemma} holds as well, picking $h_g$ to be trivial. Hence the theorem follows.  
\end{proof}


		
\begin{remark}
 Notice that to prove Theorem \ref{thm:hyp} we have applied the general technical Lemma \ref{l: technical_lemma}. In the case of hyperbolic groups the conjugating elements $h_g$ from Lemma \ref{l: technical_lemma} (\ref{def_set}) do not occur, while they were ubiquitous when proving our analogous results for RAAGs. The role of $h_g$ is different for the two classes of groups due to the structure of the centralisers: in a RAAG any centraliser is the conjugate of the centraliser of the corresponding cyclically reduced element (see Theorem \ref{t: centralizers_RAAGS}), while in hyperbolic groups the centralisers have a much nicer form, as described in Lemma \ref{l: centralizers_in_hyp_groups}.
\end{remark}

The results of this section can be generalised to relatively hyperbolic groups, in particular to any group $G$ that is torsion-free and hyperbolic with respect to virtually abelian (parabolic) subgroups. To start with, work of Dahmani \cite{DahmaniIJM09} shows that such groups have decidable Diophantine Problem, so it is interesting to ask whether any of the extensions of $\mc{DP}(G)$ are decidable or not. One can follow the strategy in this section to show that $\mc{DP}(G, \ab)$ is undecidable if $b_1(G)\geq 2$. First of all, one can find a generating set consisting entirely of hyperbolic elements (that is, elements that are not conjugate to any of the parabolic subgroup elements) by following the same arguments as in the proof of Lemma \ref{lem:basis_hyp}. Since hyperbolic elements have virtually cyclic centralisers \cite{Osin}, Lemmas \ref{l: PE-interpret_powers} through \ref{l: condition_2_for_hyperbolic} apply; one can add a pair of \ap \ elements that are hyperbolic to the basis, as in the proof of Theorem \ref{thm:hyp}, to get that $\mc{DP}(G, \ab)$ is undecidable.
This will apply, in particular, to limit groups, which are torsion-free and relatively hyperbolic with respect to free abelian subgroups.

\section{Conclusions and open questions}\label{sec:conclusions}

Motivated by the the vast literature on word equations with length constraints in the context of free semigroups, in this paper we started a systematic study of the Diophantine Problem with length constraints in groups, focusing on the constraints that require the solutions to satisfy a system of equations in the abelianisation of the group; this is akin to imposing Parikh - type constraints or adding abelian predicates to the existential theory of semigroups. 

There are clear parallels between the decidability results for semigroups versus groups. For example, the proofs that the Diophantine Problem with abelianisation constraints ($\mc{DP}, \ab$) is undecidable follow a similar pattern for free semigroups and groups. In both cases we follow the strategy of B\"uchi and Senger \cite{RichardBuchi1988} to encode the ring of integers into ($\mc{DP}, \ab$), phrasing this into the more elegant, abstract and general language of interpretability, which then allows us to deal with classes of groups beyond the free ones. We hope the framework of interpretability and further tools from model theory can lead to further applications and answer some of the questions below. 

\medskip
\noindent {\bf Problem 1.} Is ($\mc{DP}, \ab$) decidable for partially commutative semigroups?

While using interpretability via equations to encode the ring of integers into the algebraic structures we consider, it became clear that the algebraic properties of semigroups often require a different approach compared to groups. A prime example is that our strategy for RAAGs was not immediately applicable to partially commutative semigroups (also called `trace monoids'). The main reason for this is the reliance of our arguments on the normality of certain subgroups, a concept not available in a monoid. 
\medskip

\noindent {\bf Problem 2.} Is ($\mc{DP}, \ab$) decidable for graph products of monoids or groups?
 
  Trace monoids and RAAGs are examples of graph products of monoids and groups, respectively. We expect that at least in the group case our `General technical lemma' (Lemma \ref{l: technical_lemma}) applies.

\medskip
\noindent {\bf Problem 3.} Is ($\mc{DP}, \ab$) decidable for hyperbolic groups with $b_1=1$?

We showed in this paper that if the abelianisation is finite, then ($\mc{DP}, \ab$) is decidable, while the opposite holds if $b_1>1$, so the rank $1$ case lies between decidability and undecidability. 

\medskip
\noindent {\bf Problem 4.} There are other intriguing groups where the abelianisation has torsion-free rank $=1$, such as the Baumslag--Solitar groups $BS(1,n)$, or the braid group on three strands $\langle a, b \mid aba=bab\rangle$ (or more generally, odd dihedral Artin groups). In all these cases $\mc{DP}$ is known to be decidable, but ($\mc{DP}, \ab$) remains open.

\section*{Acknowledgements}
\noindent The first named author was partially supported by EPSRC Standard Grant EP/R035814/1.

The second named author was supported by the ERC grant PCG-336983, by the Basque Government  grant IT974-16,  by the Ministry of Economy, Industry and Competitiveness of the Spanish Government Grant MTM2017-86802-P, and by the  Basque Government
through the BERC 2018-2021 program and by the Ministry of
Science, Innovation and Universities: BCAM Severo Ochoa
accreditation SEV-2017-0718.

 
\bibliographystyle{plainurl}
\bibliography{references}

\end{document}